\documentclass[final,12pt]{elsarticle}
\usepackage{srcltx}
\usepackage{physics}
\usepackage{float}
\usepackage{eurosym}
\usepackage{mathtools}
\usepackage{amsmath}
\usepackage{cases}
\usepackage{amsfonts}
\usepackage{amssymb}
\usepackage{amsthm}
\usepackage{graphicx}
\usepackage{mathrsfs}
\usepackage{xcolor}
\usepackage{xpatch}
\usepackage{exscale}
\usepackage{latexsym}
\usepackage{tikz}
\usepackage{caption}
\usepackage{lineno}
\xpatchcmd{\MaketitleBox}{\hrule}{}{}{}
\xpatchcmd{\MaketitleBox}{\hrule}{}{}{}
\usepackage[colorlinks,plainpages=true,pdfpagelabels,hypertexnames=true,colorlinks=true,pdfstartview=FitV,linkcolor=blue,citecolor=red,urlcolor=black]{hyperref}
\PassOptionsToPackage{unicode}{hyperref}
\PassOptionsToPackage{naturalnames}{hyperref}
\usepackage{enumerate}
\usepackage[shortlabels]{enumitem}
\usepackage{bookmark}
\usepackage{wasysym}
\usepackage{esint}
\usepackage[ddmmyyyy]{datetime}
\usepackage[margin=2cm]{geometry}
\makeatletter
\g@addto@macro\normalsize{%
  \setlength\abovedisplayskip{4pt}
  \setlength\belowdisplayskip{4pt}
  \setlength\abovedisplayshortskip{4pt}
  \setlength\belowdisplayshortskip{4pt}
}
\numberwithin{equation}{section}
\everymath{\displaystyle}
\usepackage[capitalize,nameinlink]{cleveref}
\crefname{section}{Section}{Sections}
\crefname{subsection}{Subsection}{Subsections}
\crefname{condition}{Condition}{Conditions}
\crefname{hypothesis}{Hypothesis}{Conditions}
\crefname{assumption}{Assumption}{Assumptions}
\crefname{lemma}{Lemma}{Lemmas}
\crefname{definition}{Definition}{Definitions}

\crefformat{equation}{\textup{#2(#1)#3}}
\crefrangeformat{equation}{\textup{#3(#1)#4--#5(#2)#6}}
\crefmultiformat{equation}{\textup{#2(#1)#3}}{ and \textup{#2(#1)#3}}
{, \textup{#2(#1)#3}}{, and \textup{#2(#1)#3}}
\crefrangemultiformat{equation}{\textup{#3(#1)#4--#5(#2)#6}}%
{ and \textup{#3(#1)#4--#5(#2)#6}}{, \textup{#3(#1)#4--#5(#2)#6}}%
{, and \textup{#3(#1)#4--#5(#2)#6}}

\Crefformat{equation}{#2Equation~\textup{(#1)}#3}
\Crefrangeformat{equation}{Equations~\textup{#3(#1)#4--#5(#2)#6}}
\Crefmultiformat{equation}{Equations~\textup{#2(#1)#3}}{ and \textup{#2(#1)#3}}
{, \textup{#2(#1)#3}}{, and \textup{#2(#1)#3}}
\Crefrangemultiformat{equation}{Equations~\textup{#3(#1)#4--#5(#2)#6}}%
{ and \textup{#3(#1)#4--#5(#2)#6}}{, \textup{#3(#1)#4--#5(#2)#6}}%
{, and \textup{#3(#1)#4--#5(#2)#6}}

\crefdefaultlabelformat{#2\textup{#1}#3}
\numberwithin{equation}{section}

\newtheorem{theorem} {Theorem}[section]

\newtheorem{lemma}{Lemma}[section]

\newtheorem{example}{Example}[section]
\newtheorem{counter example}{Counter Example}[section]
\newtheorem{remark} {Remark}[section]
\newtheorem{definition} {Definition}[section]

\def\CC{{\rm \kern.24em \vrule width.02em height1.4ex depth-.05ex \kern-.26emC}}

\def\TagOnRight

\def\AA{{it I} \hskip-3pt{\tt A}}

\def\QQ{\rlap {\raise 0.4ex \hbox{$\scriptscriptstyle |$}} {\hskip -0.1em Q}}

\makeatletter
\newcommand{\vo}{\vec{o}\@ifnextchar{^}{\,}{}}
\makeatother

\def\YYint#1#2#3{{\setbox0=\hbox{$#1{#2#3}{\iint}$}
    \vcenter{\hbox{$#2#3$}}\kern-.50\wd0}}

\def\XXint#1#2#3{{\setbox0=\hbox{$#1{#2#3}{\int}$}
    \vcenter{\hbox{$#2#3$}}\kern-.50\wd0}}

\makeatletter
\def\namedlabel#1#2{\begingroup
   \def\@currentlabel{#2}%
   \label{#1}\endgroup
}
\makeatother
\makeatletter
\newcommand{\rmh}[1]{\mathpalette{\raisem@th{#1}}}
\newcommand{\raisem@th}[3]{\hspace*{-1pt}\raisebox{#1}{$#2#3$}}
\makeatother

\newcommand{\descref}[2]{\hyperref[#1]{\textnormal{\textcolor{black}{(}\textcolor{blue}{\bf #2}\textcolor{black}{)}}}}

\newcommand{\dref}[2]{\hyperref[#1]{\textcolor{black}{(}\textcolor{blue}{\bf #2}\textcolor{black}{)}}}
\newcommand{\be} {\begin{eqnarray}}
\newcommand{\ee} {\end{eqnarray}}
\newcommand{\Bea} {\begin{eqnarray*}}
\newcommand{\Eea} {\end{eqnarray*}}

\newcommand{\p}  {\prime}

\newcommand{\la} {\lambda}

\newcommand{\f}{\infty}

\newcommand{\R}{\mathbb{R}}

\newcommand{\sgn}{\mathop\mathrm{sgn}}

\newcommand{\pa}{\partial}

\newcommand{\D}{\Delta}
\newcommand{\Z}{\mathbb{Z}}
\newcommand{\sumi}{\sum_{i \in \Z}}
\newcommand{\BV}{\textrm{BV}}
\newcommand{\mM}{\mathcal{M}}
\newcommand{\jph}{j+1/2}
\newcommand{\jmh}{j-1/2}
\newcommand{\iph}{i+1/2}
\newcommand{\imh}{i-1/2}
\newcommand{\nph}{n+1/2}
\newcommand{\mP}{\mathcal{P}}
\newcommand{\mQ}{\mathcal{Q}}
\newcommand{\mK}{\mathcal{K}}
\newcommand{\ms}{\mathcal{S}}

\newcommand{\mG}{\mathcal{G}}
\DeclareMathOperator{\spt}{supp}

\DeclareMathOperator{\loc}{loc}

\DeclareMathOperator{\TV}{TV}
\DeclareMathOperator{\TVD}{TVD}
\DeclareMathOperator{\CFL}{CFL}

\newcommand{\norma}[1]{{\left\|#1\right\|}}

\newcounter{whitney}
\refstepcounter{whitney}

\newcounter{ineqcounter}
\refstepcounter{ineqcounter}
\makeatletter
\def\ps@pprintTitle{%
\let\@oddhead\@empty
\let\@evenhead\@empty
\def\@oddfoot{}%
\let\@evenfoot\@oddfoot}
\makeatother
\usepackage[doublespacing]{setspace}
\usepackage[titletoc,toc,page]{appendix}

\usepackage{pgf,tikz}
\usetikzlibrary{arrows}
\usetikzlibrary{decorations.pathreplacing}
\begin{document}

\begin{abstract}
This article concerns a scalar multidimensional conservation law where
the flux is of Panov type and may contain spatial discontinuities. We define a notion
of entropy solution and prove that entropy solutions are unique. We propose a Godunov-type finite volume scheme and prove that the Godunov approximations converge to an entropy solution, thus establishing existence of entropy solutions. We also show that our numerical scheme converges at an optimal rate of $\mathcal{O}(\sqrt{\D t}).$	To the best of our knowledge, convergence of the Godunov type methods in multi-dimension and  error estimates of the numerical scheme in one as well as in several dimensions are the first of it's kind for conservation laws with discontinuous flux.
We present numerical examples that illustrate the theory.
	\end{abstract}
\begin{frontmatter}

\title{A Godunov type scheme and error estimates for multidimensional scalar conservation laws with Panov-type discontinuous flux}
\author[myaddress1]{Shyam Sundar Ghoshal}
\ead{ghoshal@tifrbng.res.in}


\address[myaddress1]{Centre for Applicable Mathematics,Tata Institute of Fundamental Research, Post Bag No 6503, Sharadanagar, Bangalore - 560065, India.}

\author[myaddress2]{John D. Towers}
\ead{john.towers@cox.net}

\author[myaddress1]{Ganesh Vaidya}
\ead{ganesh@tifrbng.res.in}

\address[myaddress2]
{MiraCosta College, 3333 Manchester Avenue, Cardiff-by-the-Sea, CA 92007-1516, USA.}
\end{frontmatter}
\tableofcontents

\section{Introduction}

In this article we study the initial value problem for the following conservation law in several dimensions, 
\begin{eqnarray}
\label{eq:discont}
u_t+\sum\limits_{i=1}^d \frac{\partial}{\partial x_i}A_i(\vb{x},u)&=& 0 \,\quad\quad \quad \quad \text{for}\,\,\,(t,\vb{x}) \in(0,\infty)\times\R^d,\\
\label{eq:data}  u(0,\vb{x})&=&u_0(\vb{x}) \,\,\quad\quad \text{for}\,\,\,\vb{x} \in\R^d, 
\end{eqnarray} where the flux $\vb{A}:\R^d \times \R \rightarrow \R^d$ is of Panov type, as in \cite{Panov2009b} and can have infinitely many spatial discontinuities with accumulation points. In particular, $\vb{A}(\vb{x},u)=\vb{g}(\beta(\vb{x},u))$, where 
 $\vb{g}$ can be a locally Lipschitz continuous real-valued function and $\beta(\vb{x},\cdot)$ is a monotone function for each $\vb{x}\in\R^d.$ Thus in this article we do not  impose any restriction on the shape of $u \mapsto \vb{A}(\vb{x},u)$ and thereby extending the one dimensional convergence analysis discussed in   \cite{GJT_2019,GTV_2020_2,JDT_2020}. 
 One-dimensional conservation laws with discontinuous flux have been the subject of a large literature over the past several decades.
The multidimensional case has received less attention, see e.g., 
\cite{aleksic_mitrovic,AM_multi,Crasta:2016aa,Crasta:2015aa,graf_et_al,
hold_karl_mit,Karlsen_Rascle_Tadmor,Panov2009a,Panov2009b}.

Even for the case of one dimension ($d=1$), mathematical analysis of these type of equations is complicated due to the presence of discontinuities in the spatial variable of the flux function $\vb{A}(\cdot,\cdot)$. It is well known that when $\vb{x} \mapsto \vb{A}(\vb{x},u)$ is not sufficiently smooth, the classical Kruzkov inequality, 
\begin{eqnarray}\label{kruzkov}
\partial_t |u-k|+\partial_{\vb{x}}\left[\sgn(u-k) (\vb{A}(\vb{x},u)-\vb{A}(\vb{x},k))\right] + \sgn(u-k) \partial_{\vb{x}}\vb{A}(\vb{x},k) \leq 0, \quad k\in \R,
\end{eqnarray}
does not make sense due to the term $\sgn(u-k){\partial_{\vb{x}}}\vb{A}(\vb{x},k).$ 
When the spatial discontinuities are discrete, the uniqueness of weak solutions is obtained by imposing certain additional conditions (known as interface entropy conditions) along the spatial discontinuities of the flux, which require the existence of traces. Various types of entropy conditions can be chosen depending on the underlying physics of  the problem, details of which can be found in \cite{AJG, AMV,mishra2007convergence,adimurthi2000conservation, andreianov2013vanishing, AKR,BGKT,burger2006extended,bkrt2,BKT_2009,towers_disc_flux_1} and the references therein. However, when the spatial discontinuities accumulate, the traces do not exist in general. To overcome this obstacle, the notion of \textit{adapted entropy solutions} has been proposed, first in  \cite{BaitiJenssen} for a monotone flux, and then in \cite{AudussePerthame} for monotone or unimodal flux.
The adapted entropy approach to uniqueness can be seen as a generalization of the classical Kruzkov theory. Adapted entropy conditions use a certain class of spatially dependent steady state solutions $k=k(\vb{x})$ chosen so that the term 
$\sgn(u-k(\vb{x})){\partial_{\vb{x}}}\vb{A}(\vb{x},k(\vb{x}))$ vanishes. This work was later generalized in \cite{Panov2009a} to  $\vb{A}(\vb{x},u)$ of the form $\vb{g}(\beta(\vb{x},u)).$ In addition, uniqueness results for solutions of \eqref{eq:discont}-\eqref{eq:data} have been further generalized to fluxes possessing degeneracy, see \cite{GTV_2020_1}. The convergence analysis of the numerical schemes for these kind of fluxes was open for a quite a long time and recently this has been answered in \cite {GJT_2019,GTV_2020_2,JDT_2020}. 

The notion of interface entropy condition was then generalized to several dimensions in \cite{AM_multi} and  existence of such solutions was established via the vanishing viscosity method. However the convergence of finite volume approximations remains open for the multidimensional problem even for the case of single discontinuity. For the case of homogeneous flux (no spatial dependence), convergence of numerical approximations is established by the so-called dimension splitting techniques see for example, \cite{CM_1980,HKLR_book}.  The classical dimensional splitting arguments cannot be used when the fluxes are discontinuous because the solutions do not satisfy the TVD property in general \cite{ADGG, Ghoshal-JDE, Ghoshal-NHM, GTV_2020_1}. 
However for certain class of  Panov-type discontinuous fluxes, recently \cite{GTV_2020_2} shows that though the solution does not satisfy TVD property, the function $\beta(\cdot,u(\cdot))$ possesses the TVD  property. So in this article we prove this property for general $g$ and use it to establish the convergence of the dimension splitting method. Our technique also implies the existence of a BV bound on the solution for the class fluxes which are under consideration, which is of independent interest.

Another aim  of this article is to study the error analysis of our numerical method. From a practical point of view, along with the convergence,  it is also important to understand how fast the scheme converges, i.e. how fast the error of approximation of the exact solution
$u$ by the numerical approximation $||u^{\Delta }(T,\cdot)-u(T, \cdot)||_{L^1}$ goes to zero as mesh size $\Delta $ goes to zero. This can be measured in terms of the $\alpha$ which satisfies the following
\begin{eqnarray}||u^{\Delta }(T,\cdot)-u(T,\cdot)||_{L^1} \leq C {\Delta t}^{\alpha}.
\end{eqnarray}
In addition, convergence rates can also be used for a posteriori error based mesh adaptation \cite{VD_2000} and optimal design of multilevel Monte Carlo methods \cite{BRK_19}.
In the case of a spatially independent flux with $d=1$, using the doubling of the variable argument, Kuznetsov \cite{Kuznetsov} proved that monotone schemes converge to the weak solution satisfying the Kruzkov entropy condition with $\alpha=1/2$.  Reference \cite{Karlsen_multi_error} shows that these results are indeed true in several spatial dimensions (for a homogeneous flux). Sabac constructed explicit examples in \cite{sabac_optimal} which imply that this estimate is optimal. 
Of late,  \cite{fjordholm2020convergence} proves the convergence rates of monotone schemes for conservation laws for Holder continuous initial data with Holder exponent greater than 1/2, where bounded variation of the initial data is not required.  For unilateral constrained problem  \cite{cances2012error} provides error estimate for the Godunov approximation of the problem to be  $\mathcal{O}({\Delta t}^{\frac{1}{3}}).$ However the rates can be shown to be the optimal rate of $\mathcal{O}(\sqrt{\Delta t})$ provided bounds on the temporal total variation of the finite volume approximation exists in the cells adjacent to the point where constraint is imposed. The techniques introduced in this paper can be adapted to scalar conservation laws with  discontinuous flux (with finitely many discontinuities) and the rate of convergence depends on the temporal total variation bounds of the finite volume approximation in the cells adjacent to the spatial of discontinuities of the flux (see section 7.3, \cite{cances2012error}). Such bounds on temporal variation can be easily obtained for Riemann data, however such bounds were not known for general data. Very recently, the  bounds on the temporal total variation of the finite volume approximation are proved for the the case of strictly monotone fluxes \cite{BR_2020} and thus the rates are shown to be 1/2 for monotone fluxes with finitely many spatial discontinuities.  These estimates are  obtained based on the idea that, for the case of monotone fluxes, problem of discontinuous flux can be treated as boundary value problem with a $\BV$ boundary data, where Kuznetsov's type arguments can be invoked and combining the boundary value problems, error estimates can be obtained for the IVP \eqref{eq:discont}-\eqref{eq:data},  which allows to estimate the boundary terms in space at the discontinuities that appear when applying the classical Kuznetsov theory to problem. 

To the best of our knowledge proofs for the optimal rate $1/2$ are not known for general $\BV$ data for non monotone flux even in the case of single discontinuity. Also, no results on error estimates are available when spatial discontinuities of the flux are allowed to be infinite, which in turn may accumulate. 
In this article,  for a certain class of fluxes we prove that Godunov type schemes converge to the adapted entropy solution with the optimal rate 1/2, thus dispensing with the assumption of  strict monotonicity and finitely many points of discontinuity of \cite{BR_2020} to obtain the optimal rate 1/2. Since the methods of \cite{BR_2020} are not applicable when the set of spatial discontinuities contains accumulation points,  we prove a Kuznetsov type lemma based on adapted entropy formulation to obtain the error estimates. To the best of our knowledge, this is the first error estimate for conservation laws with discontinuous flux, where the set of spatial discontinuities of $A_i(\vb{x},u)$ is infinite and may also contain accumulation points.

 In Section~\ref{sec_uniqueness} we define the relevant notion of solution of entropy solution and prove uniqueness 
of entropy solutions.
Section~\ref{sec_godunov} describes the Godunov-type finite volume scheme we use to prove existence. We prove convergence
of the Godunov approximations, first in the one-dimensional case (Section~\ref{sec_Godunov_1D}), and then in the 
multidimensional case (Section~\ref{sec_Godunov_multiD}). The convergence result, combined with the uniqueness result of
Section~\ref{sec_uniqueness}, yields a well-posedness result for the problem.
Section~\ref{sec_error_est} establishes of rate convergence estimate, obtained by a Kuznetsov-type analysis.
In Section~\ref{sec_numerical} we present the results of several numerical experiments.

\section{Uniqueness of the adapted entropy solution in several dimensions}\label{sec_uniqueness}
We consider the fluxes of the form $\vb{A}(\vb{x},u)=\vb{g}(\beta(\vb{x},u))$, where $\vb{g}$ and $\beta$ satisfy the following assumptions.
\begin{enumerate}[label=\textbf{A-\arabic*}]
\item \label{A1} $\vb{g}: \R \rightarrow \R^d$ is a  locally Lipschitz continuous function. 
\item \label{A2} $\beta(\vb{x},u)$ is continuous on $\prod_{i=1}^{d}\mathbb{R}\setminus {\Omega}_i \times \mathbb{R},$ where $\Omega_i$ for $i=1,2,\ldots, d$ are closed zero measure sets in $\R$. In addition $u \mapsto \beta(\vb{x},u)$ is strictly increasing. 
\item \label{A3} For  strictly increasing functions  $h_1,h_2:\R \rightarrow \R$ such that $\lim\limits_{|u| \rightarrow \infty}\abs{h_1(u)}=\infty,$ for any fixed $u$, $h_1(u)\leq \beta(\vb{x},u)\leq h_2(u),$ for all $\vb{x}\in \R^d.$  \end{enumerate}
\begin{definition}[Adapted Entropy Condition] \label{def_adapted_entropy}
Let $Q= [0,T)\times\R^d$. 
\begin{eqnarray}\label{E1}
{\partial_t} |u(t,\vb{x})-k_{\alpha}(\vb{x})| +\sum\limits_{i=1}^d{\partial_{x_i}}\left[ \sgn (u-k_{\alpha}(\vb{x})) (A_i(u,\vb{x})-g_i(\alpha)) \right] \leq 0, \text{ in } \mathcal{D}'{(Q)}
\end{eqnarray}for $\alpha \in \R.$
Or equivalently, for all $0\leq\phi \in C_c^{\infty}(Q),$
\begin{align}
&\int\limits_Q |u(t,\vb{x})-k_{\alpha}(\vb{x})|\phi_t(t,\vb{x})+\sum\limits_{i=1}^d\sgn (u(t,\vb{x})-k_{\alpha}(\vb{x})) (A_i(\vb{x},u(t,\vb{x}))-g_i(\alpha))\phi_{x_i}(t,\vb{x})\, d\vb{x} dt\nonumber \\
&+\int\limits_{\R}|u_0(\vb{x})-k_{\alpha}(\vb{x})|\phi(0,\vb{x})\,d\vb{x}\geq 0,\label{E2}
\end{align}
where $k_{\alpha}: \R^d \rightarrow \R$ is a stationary state defined by $k_{\alpha}(\vb{x}):=\beta^{-1}(\vb{x},\alpha).$
\end{definition}
\begin{remark}\normalfont For $d=1$ and $\vb{A}(\vb{x},u)$  unimodal, the above definition of \textit{adapted entropy solutions} can be viewed as the generalization of the definition given in \cite{AudussePerthame}, in the following sense:\\
Let $\Psi_{\vb{A}}(\vb{x},u)$ denote the  singular map corresponding to $\vb{A}(\vb{x},u).$ Then the flux can be written in the Panov form
$\vb{A}(\vb{x},u)=\vb{g}(\beta(\vb{x},u)),$ with $\vb{g}(u)=\abs{u}$ and $\beta(\vb{x},u)=\Psi_{\vb{A}}(\vb{x},u).$ Now, for $\alpha \in \R,$ we have, 
\begin{equation*}
k_{\alpha}(\vb{x})= 
\begin{cases}
k_{\alpha}^+(\vb{x}), \quad & \alpha\geq 0,\\
k^-_{\abs{\alpha}}(\vb{x}), \quad & \alpha\leq 0.
\end{cases} 
\end{equation*}
Here,
$k^{\pm}_{\alpha}(\vb{x}):=(A^{\pm})^{-1}(\vb{x},\alpha)$ for $ \alpha >0.$
\end{remark}

\begin{theorem}\label{theorem1}
Let $u,v \in C([0,T];L^1_{loc}(\R^d)) \cap L^{\infty}(Q)$ be entropic solutions to the IVP \eqref{eq:discont}-\eqref{eq:data} with initial data $u_0,v_0 \in L^{\infty}(\R).$ Assume the flux satisfies the hypothesis \eqref{A1}--\eqref{A3}. Then for $t\in [0,T]$ the following holds,
\begin{equation*}
\displaystyle \int\limits_{S_0}|u(t,\vb{x})-v(t,\vb{x})|d\vb{x} \leq \displaystyle\int\limits_{S_t}|u_0(\vb{x})-v_0(\vb{x})|d\vb{x},
\end{equation*}
where $\overline{M}:=\sup\{\abs{\vb{A}_u(\vb{x},u(t,\vb{x}))};\,\vb{x}\in\R^d,0\leq t\leq T\}$, $S_0=\prod_{i=1}^{d}[a_i,b_i]$,   $S_t=\prod_{i=1}^{d}[a_i-\overline{M}t,b_i+\overline{M}t],$ $i=1,2,\ldots ,d$ and  
 $-\f\leq a_i<b_i\leq \f$ .
\end{theorem}
\begin{proof}
Let $\xi_{\eta},\rho_{\epsilon}\in C^{\infty}_c(\R)$  be mollifiers, such that  $supp(\rho) \subset [-2,-1].$ We define $\Phi_{\eta,\epsilon}: {Q}^2\rightarrow \R \in C_c^{\infty}({Q}^2) $ as follows,
 \begin{equation*}
 \Phi_{\eta,\epsilon}(t,\vb{x},s,\vb{y})=\phi(t,\vb{x})\rho_{\epsilon}(t-s)\prod_{i=1}^{d}\xi_{\eta}(x_i-y_i).
\end{equation*}
Set $k_{\beta(\vb{y},v(s,\vb{y}))}(\vb{x})=\tilde{v}(s,\vb{y},\vb{x}),$ by the definition of $k_{\alpha}$ we get  
\begin{equation}\label{4.6}
\pa_t|u(t,\vb{x})-\tilde{v}(s,\vb{y},\vb{x})|+\sum_{i=1}^d\pa_{x_i}\left[ \sgn (u(t,\vb{x})-\tilde{v}(s,\vb{y},\vb{x})) (A_i(\vb{x},u(t,\vb{x}))-A_i(\vb{y},v(s,\vb{y})))\right]\le 0.
\end{equation}
Similarly, now rewriting the entropy condition $v(s,\vb{y})$
with $\alpha=\beta(\vb{x},u(t,\vb{x}))$, we get
\begin{eqnarray}\label{4.7}
\pa_s |v(s,\vb{y})-\tilde{u}(t,\vb{x},\vb{y})|+\sum_{i=1}^d\pa_{y_i}\left[ \sgn (v(s,\vb{y})-\tilde{u}(t,\vb{x},\vb{y})) (A_i(\vb{y},v(s,\vb{y}))-A_i(\vb{x},u(t,\vb{x})))\right]\leq 0.
\end{eqnarray}
Integrating \eqref{4.6} in $\vb{x},\vb{y},t,s$ against the function $\Phi_{\eta,\epsilon}(t,\vb{x},s,\vb{y})$, we have,
\begin{align*}
&\int\limits_{{Q}^2} |u(t,\vb{x})-\tilde{v}(s,\vb{y},\vb{x})|\phi(t,\vb{x})\rho^{\p}_{\epsilon}(t-s) \prod_{i=1}^{d}\xi_{\eta}(x_i-y_i)d\vb{x} d\vb{y} dt ds\\
+&\int\limits_{{Q}^2} |u(t,\vb{x})-\tilde{v}(s,\vb{y},\vb{x})|\phi_t(t,\vb{x})\rho_{\epsilon}(t-s)\prod_{i=1}^{d}\xi_{\eta}(x_i-y_i)d\vb{x} d\vb{y} dt ds\\
+&\sum_{i=1}^d\int\limits_{{Q}^2}\Big\{\left[ \sgn (u(t,\vb{x})-\tilde{v}(s,\vb{y},\vb{x})) (A_i(\vb{x},u(t,\vb{x}))-A_i(\vb{y},v(s,\vb{y})))\right]\phi(t,\vb{x})\rho_{\epsilon}(t-s) \xi^{\p}_{\eta}(x_i-y_i)\\
&\quad \quad \quad \quad \quad \quad \quad \quad \quad \quad \quad \quad \quad \quad \quad \quad \quad \quad\quad \quad \quad \quad \quad \quad \quad \quad \quad\quad \quad \quad \times \prod_{j\neq i}\xi_{\eta}(y_j-y_j) \Big\}d\vb{x} d\vb{y} dt ds\\
+&\sum_{i=1}^d\int\limits_{{Q}^2}\left[ \sgn (u(t,\vb{x})-\tilde{v}(s,\vb{y},\vb{x})) (A_i(\vb{x},u(t,\vb{x}))-A_i(\vb{y},v(s,\vb{y})))\right]\phi_{x_i}(t,\vb{x})\rho_{\epsilon}(t-s) \prod_{i=1}^{d}\xi_{\eta}(x_i-y_i)) d\vb{x} d\vb{y} dt ds \\
+&\int_{Q}\int\limits_{\R}|u_0(\vb{x})-\tilde{{v}}(s,\vb{y},\vb{x})|\phi(0,\vb{x})\rho_{\epsilon}(-s)\prod_{i=1}^{d}\xi_{\eta}(x_i-y_i) d\vb{x} d\vb{y} ds\geq 0.
\end{align*}
Integrating \eqref{4.7} in $\vb{x},\vb{y},t,s$ against function $\Phi_{\eta,\epsilon}(t,\vb{x},s,\vb{y})$, we get
\begin{align*}
&-\int\limits_{{Q}^2} |v(s,\vb{y})-\tilde{u}(t,\vb{x},\vb{y})|\phi(t,\vb{x})\rho^{\p}_{\epsilon}(t-s)\prod_{i=1}^{d}\xi_{\eta}(x_i-y_i)d\vb{x} d\vb{y} dt ds\\
&-\sum\limits_{i=1}^d\int\limits_{{Q}^2}\left[ \sgn (v(s,\vb{y})-\tilde{u}(t,\vb{x},\vb{y})) (A_i(\vb{y},v(s,\vb{y}))-A_i(\vb{x},u(t,\vb{x})))\right]\phi(t,\vb{x})\rho_{\epsilon}(t-s) \xi^{\p}_{\eta}(x_i-y_i)\\
&\quad \quad \quad \quad \quad \quad \quad \quad \quad \quad \quad \quad \quad \quad \quad \quad \quad \quad\quad \quad \quad \quad \quad \quad \quad \quad \quad\quad \quad \quad \times \prod_{j\neq i}\xi_{\eta}(y_j-y_j) d\vb{x} d\vb{y} dt ds\\
+&\int_{Q}\int\limits_{\R}|v_0(\vb{x})-\tilde{{u}}(t,\vb{x},\vb{y})|\phi(t,\vb{x})\rho_{\epsilon}(t)\prod_{i=1}^{d}\xi_{\eta}(x_i-y_i) d\vb{x} d\vb{y} dt\geq 0.
\end{align*}

Adding the above two inequalities and collecting the common terms, we have the sum of the following 6 terms:
\begin{enumerate}[label=\roman*.]
\item 
\begin{equation*}
\int\limits_{{Q}^2} |u(t,\vb{x})-\tilde{v}(s,\vb{y},\vb{x})|\phi_t(t,\vb{x})\rho_{\epsilon}(t-s)\prod_{i=1}^{d}\xi_{\eta}(x_i-y_i)d\vb{x} d\vb{y} dt ds,
\end{equation*}
 \item 
 \begin{equation*}
 \int\limits_{{Q}^2} \left(|u(t,\vb{x})-\tilde{v}(s,\vb{y},\vb{x})|-|v(s,\vb{y})-\tilde{u}(t,\vb{x},\vb{y})|\right)\phi(t,\vb{x})\rho^{\p}_{\epsilon}(t-s)\prod_{i=1}^{d}\xi_{\eta}(x_i-y_i)d\vb{x} d\vb{y} dt ds,
 \end{equation*}
\item 
\begin{equation*}
\sum\limits_{i=1}^d\int\limits_{{Q}^2}\left[ \sgn (u(t,\vb{x})-\tilde{v}(s,\vb{y},\vb{x})) (A_i(\vb{x},u(t,\vb{x}))-A_i(\vb{y},v(s,\vb{y})))\right]\phi_{x_i}(t,\vb{x})\rho_{\epsilon}(t-s) \prod_{i=1}^{d} \xi_{\eta}(x_i-y_i) d\vb{x} d\vb{y} dt ds,
\end{equation*}
\item 
\begin{align*}
&\sum\limits_{i=1}^d\int\limits_{{Q}^2}\left[\Big(\sgn (v(s,\vb{y})-\tilde{u}(t,\vb{x},\vb{y}))+\sgn (u(t,\vb{x})-\tilde{v}(s,\vb{y},\vb{x}))\Big) (-A_i(\vb{y},v(s,\vb{y}))+A_i(\vb{x},u(t,\vb{x})))\right]\\
&\quad \quad\quad \quad\quad \quad\quad \quad\quad \quad\quad \quad\quad \quad \quad\quad \quad\quad \quad\quad  \quad
\phi(t,\vb{x})\rho_{\epsilon}(t-s) \frac{\partial}{\partial x_i}\prod_{j=1}^{d}\xi_{\eta}(y_j-y_j) d\vb{x} d\vb{y} dt ds,
 \end{align*}
\item 
\begin{equation*}
\int\limits_{Q}\int\limits_{\R}|u_0(\vb{x})-\tilde{{v}}(s,\vb{y},\vb{x})|\phi(t,\vb{x})\rho_{\epsilon}(-s)\prod_{i=1}^{d}\xi_{\eta}(x_i-y_i) d\vb{x} d\vb{y} ds,
\end{equation*}
\item 
\begin{equation*}
\int\limits_{Q}\int\limits_{\R}|v_0(\vb{x})-\tilde{{u}}(t,\vb{x},\vb{y})|\phi(t,\vb{x})\rho_{\epsilon}(t)\prod_{i=1}^{d}\xi_{\eta}(x_i-y_i) d\vb{x} d\vb{y} dt,
\end{equation*}
\end{enumerate}
is greater than or  equal to 0.
Now the rest of the proof can be completed on the similar lines of \cite{AudussePerthame} using the following properties of $\tilde{u}$ and $\tilde{v}:$
\begin{eqnarray*}
\tilde{v}(s,\vb{y},\vb{y})=v(s,\vb{y}) \text{ and }
\tilde{u}(t,\vb{x},\vb{x})=u(t,\vb{x}) \text{ for } \vb{x},\vb{y} \in \R^d, t>0,
\end{eqnarray*}
\begin{eqnarray*}
\tilde{v}(s,\vb{y},\vb{x}) &\rightarrow& v(s,\vb{y}), \text{ as } \vb{x} \rightarrow \vb{y} \text{ for a.e. } \vb{y}\in \R^d,\\
\tilde{u}(t,\vb{x},\vb{y}) &\rightarrow& u(t,\vb{x}), \text{ as } \vb{y} \rightarrow \vb{x} \text{ for a.e. } \vb{x}\in \R^d,
\end{eqnarray*}
 \begin{eqnarray*}
\left(\sgn (v(s,\vb{y})-\tilde{u}(t,\vb{x},\vb{y}))+\sgn (u(t,\vb{x})-\tilde{v}(s,\vb{y},\vb{x}))\right) =0.
\end{eqnarray*}\end{proof}
\begin{remark}\normalfont
The notion of adapted entropy can be generalized for $\beta=(\beta_1,\beta_2,\ldots,\beta_d)$ satisfying the  assumptions \eqref{A2}-\eqref{A3} if in addition for every  $\alpha_1 \in \R,$ there exists a unique $(d-1)$-tuple $(\alpha_2,\alpha_3,\ldots,\alpha_d)$ such that  the following holds:
\begin{eqnarray*}\beta_i(\vb{x},\beta_1^{-1}(\vb{x},\alpha_1)) =\alpha_i \quad \text{for } \vb{x}\in \R^d. \end{eqnarray*}
Uniqueness of the solutions can be proved on the similar lines of Theorem \ref{theorem1} with appropriate changes. 
\end{remark}

\section{Godunov scheme and its convergence}\label{sec_godunov}

\subsection{Convergence in one dimension }\label{sec_Godunov_1D}
We briefly present the the convergence analysis for a general $g.$ Most of the proofs are in the spirit of \cite{GTV_2020_2}. Consider the initial value problem \eqref{eq:discont}-\eqref{eq:data}, where in addition the flux $\vb{A}(\vb{x},u)=A(x,u)=g(\beta(x,u))$ satisfies the following:
\begin{enumerate}[label=\textbf{B-\arabic*}]
\item \label{B1}For $u,v \in [-r,r],$
\begin{equation}\label{beta_2}
\abs{\beta(x,v)-\beta(x,u)} \le \mK_1(r) \abs{u-v},
\end{equation}
for some continuous $\mK_1: \R \rightarrow [0,\infty)$.
Also,
\begin{equation}\label{beta_1A}
\abs{\beta(x,u)-\beta(y,u)} \le \mK_2(u) \abs{\alpha(x)-\alpha(y)},
\end{equation}
where $\mK_2: \R \rightarrow [0,\infty)$ is continuous and $\alpha \in \BV(\R)$.
\item\label{B2} For some $\mK_3 >0$, independent of $x$,
\begin{equation}\label{beta_1}
\abs{\beta(x,u)-\beta(x,v)} \ge \mK_3 \abs{u-v}.
\end{equation}\item \label{B3} $g(z)$ is (locally) Lipschitz-continuous, i.e.,
\begin{equation}\label{g_lip}
\textrm{$\abs{g(z_1) - g(z_2)} \le \mK_4(M) \abs{z_1-z_2}$ for $z_1, z_2 \in [-M,M]$}, M>0,
\end{equation}
where $\mK_4:\R \rightarrow [0,\infty)$ is continuous.

\end{enumerate}
For $\D x,\D t>0,$ consider equidistant spatial grid points $x_i:=i\D x$ for $i\in\Z$ and temporal grid points $t^n:=n\D t$ 
for integers $0 \le n\le N$, such that $T \in [t^N,t^{N+1})$. Let $\la:=\D t/\D x$. Let $\chi(x)$ denote the indicator function of $C_i:=[x_i - \D x /2, x_i + \D x /2)$, and let
$\chi^n(t)$ denote the indicator function of $C^n:=[t^n,t^{n+1})$. We approximate the initial data according to:
\begin{equation}
u^{\D}_0(x):=\sumi\chi(x)u^0_i\quad \mbox{where }u^0_i=u_0(y_i)\mbox{ for }i\in\Z.
\end{equation}
The  approximations generated by the scheme are denoted by $u_i^n$, where $u_j^n \approx u(x_j,t^n)$.
The grid function $\{u_i^n\}$ is extended to a function defined on $\Pi_T=\R \times [0,T]$ 

 via
\begin{equation*}\label{def_u_De}
u^{\D}(x,t) =\sum_{n=0}^N \sumi \chi(x) \chi^n(t) u_i^n.
\end{equation*}
Similarly, we define another grid function $\beta_i^n=\beta(x_i,u_i^n)\approx \beta(x_i,u(x_i,t^n)) ,$ and is extended to a function defined on $\Pi_T$  via \begin{equation*}\label{def_beta_De}
\beta^{\D}(x,t) =\sum_{n=0}^N \sumi \chi(x) \chi^n(t) \beta_i^n.
\end{equation*}
We use the symbols $\Delta_{\pm}$ to denote spatial difference operators:
\begin{equation}
\Delta_+ z_i = z_{i+1}-z_i, \quad \Delta_- z_i = z_{i}-z_{i-1}.
\end{equation}For a sequence $\{a_i\}_{i\in \Z},$ we define the total variation by
 \begin{eqnarray*}
 \TV(a):=\sum\limits_{i\in \Z}\abs{a_i-a_{i-1}}.
\end{eqnarray*}

We use the Godunov type scheme given by:
\begin{equation}\label{scheme_A}
u_i^{n+1} = u_i^n - \lambda \D_- \bar{A}(u^n_i,u^n_{i+1},x_i,x_{i+1}), \quad i \in \Z,  n=0,1,2,\ldots,
\end{equation}
where the numerical flux $\bar{A}$ is the generalized Godunov flux of \cite{GTV_2020_2}:
\begin{eqnarray}\label{def_bar_A_direct}
\bar{A}(u,v,x_i,x_{i+1}) :=\bar{g}\left(\beta(x_i,u), \beta(x_{i+1},v)\right)
\end{eqnarray}
and 

\begin{equation}
\bar{g}(p,q) = 
\begin{cases}
\min_{w \in [p,q]}g(w), \quad & p\le q,\\
\max_{w \in [q,p]}g(w), \quad & p\ge q.
\end{cases} 
\end{equation}
$\bar{A}$ is a generalization of the classical Godunov numerical flux \cite{CranMaj:Monoton,leveque_book} with $\beta(x,u)=u$ in the sense that
\begin{equation}
\bar{A}(u,v,x,x) = 
\begin{cases}
\min_{w \in [u,v]} A(x,w), \quad & u\le v,\\
\max_{w \in [v,u]} A(x,w), \quad & u\ge v.
\end{cases} 
\end{equation}
\begin{lemma}The following bounds hold:
\begin{enumerate}[label=\roman*.]
\item $
\overline{\alpha}_+:=\sup_{x\in \R}\beta(x,u_0(x)) < \infty$  and $
\overline{\alpha}_-:=\inf_{x\in \R}\beta(x,u_0(x))>-\infty.$

\item There exists $\mM$ such that \begin{eqnarray}
||k_{\overline{\alpha}_{\pm}}||_{L^{\infty}}<\mM.
\end{eqnarray}
\end{enumerate}
\end{lemma}
\begin{proof}
Proof follows due to assumption \eqref{A3}.
\end{proof}
\begin{remark}\normalfont The above lemma is  the analogue of Lemma 3.1 of \cite{GJT_2019} for $k_{\overline{\alpha}}^{\pm}.$
\end{remark} Let $\ms = \sup_{\abs{u} \le \mM, x\in \R} \abs{\beta(x,u)}$, and
define
$
L_{\beta} = \mK_1(\mM),  L_g = \mK_4(\ms).
$
Hereafter the ratio $\lambda=\frac{\D t}{\D x}$ is fixed and satisfies the $\CFL$ condition:
\begin{eqnarray}\label{cfl}\lambda L_{g} L_{\beta} \leq 1/2.\end{eqnarray}

	\begin{lemma}\label{lemma:A_u_bound_multi}
                  Under the $\CFL$ condition \eqref{cfl}, the scheme is monotone and the Godunov approximations are bounded:\begin{equation}\label{u_bounded_multi}
	\abs{u_i^n} \le \mM, \quad i \in \Z, n=0,1,2,\ldots.
	\end{equation}
\end{lemma}
\begin{proof}
Monotonicity follows because $\overline{g}$ is a monotone numerical flux and $\beta(x,\cdot)$ is increasing. For the bound on the approximations, note that $k_{\overline{\alpha}_{\pm}}$ are steady states and thus proof can be completed in the spirit of Lemma 3.5 and Lemma 3.6 of \cite{GJT_2019}. 
\end{proof}

\begin{lemma}Under the $\CFL$ condition \eqref{cfl}, the following properties hold: 
\begin{enumerate}[label=\roman*.]
\item \label{lemma_time_continuity}
Discrete time continuity estimates:

\begin{eqnarray}\label{time_cont}
\sumi \abs{u_i^{n+1}-u_i^n} \le \mK_5 \TV(\beta^0), n=0,1,2,\ldots
\end{eqnarray}

where $\mK_5>0$ is independent of the mesh size $\D$.
\item\label{beta_tvd} $\TVD$ property with respect to $\{\beta_i^n\}:$ 
    \begin{eqnarray}
    \sum\limits_{i\in \mathbb{Z}}\abs{\beta_{i+1}^{n+1} - \beta_i^{n+1}} \le \sum\limits_{i\in \mathbb{Z}}\abs{\beta_{i+1}^{n} - \beta_i^{n}}.
    \end{eqnarray}
\item \label{contractivity} Discrete $L^1$ contractivity: Let $u_0,v_0 \in L^{\infty} \cap L^1(\R)$ and $\{u_i^n\},\{v_i^n\}$ be the corresponding numerical approximations calculated by the Godunov scheme. Then,\begin{eqnarray} \sum\limits_{i\in \mathbb{Z}}\left|u_i^{n+1}-v_i^{n+1}\right|\leq \sum\limits_{i\in \mathbb{Z}}\left|u_i^n-v_i^n\right|\quad n=0,1,2,\ldots.
\end{eqnarray}
\item\label{dis_entropy} Discrete entropy inequality:
	\begin{equation}\label{ent_discrete_A}
	\abs{u^{n+1}_i- k^{\alpha}_{i}} 
	\le \abs{u_i^{n} - k^{\alpha}_{i}}
	- \lambda (\mP^n_{\iph} - \mP^n_{\imh}),\mbox{ for all }i\in\Z,n=0,1,2,\ldots,
	\end{equation}
	where
	\begin{equation*}
	\mathcal{P}^n_{\iph} = \bar{A}(u_i^n \vee k^{\alpha}_{i},u_{i+1}^n \vee k^{\alpha}_{i+1},x_i,x_{i+1})	
	                                 -  \bar{A}(u_i^n \wedge k^{\alpha}_{i},u_{i+1}^n \wedge k^{\alpha}_{i+1},x_{\Delta x},x_{i+1}).
	\end{equation*}

\end{enumerate}
\end{lemma}
\begin{proof}
The proofs of \eqref{lemma_time_continuity}, \eqref{beta_tvd} and \eqref{dis_entropy} are same as the ones presented in \cite{GTV_2020_2}. The proof of \eqref{contractivity} follows from the Crandall-Tartar lemma \cite{HRbook}.
\end{proof}

\begin{theorem}\label{theorem2}
	Assume that the flux function $A(x,u)$ satisfies   Assumptions \descref{B1}{B-1} through \descref{B3}{B-3}, 
	and that $u_0\in \BV(\R)$. 
	Then as the mesh size $\D \rightarrow 0$, the approximations $u^{\D}$ generated by the Godunov scheme described 
	above converge in $L^1_{\loc}(Q)$ and pointwise a.e. in $Q$ to the unique adapted entropy solution 
	$u \in L^{\infty}(Q) \cap C([0,T]:L^1_{\loc}(\R))$ corresponding to the Cauchy problem \eqref{eq:discont}, 
	\eqref{eq:data} with initial data $u_0$. In addition, the total variation $u(\cdot,t)$ is uniformly  bounded for $t\geq 0$.
\end{theorem}
\begin{proof}
The proof is same as the one presented in \cite{GTV_2020_2}.
\end{proof}
\subsection{Convergence in several dimensions} \label{sec_Godunov_multiD}
Now, we give the proof of convergence of the numerical scheme to the adapted entropy solution.  For the sake of simplicity we assume $d=2,$ but the proof carries over for the higher dimensions as well in the same way.
We additionally assume that the fluxes satisfy the following:
\begin{enumerate}[label=\textbf{C-\arabic*}]
 \item \label{C1} $\vb{g}(\vb{z})$ is (locally) Lipschitz-continuous., i.e., for $i=1,2$
 \begin{equation}\label{h_lip}
 \abs{g_i(z_1) - g_i(z_2)} \le \mK_6(M) \abs{z_1-z_2} \text{ for } z_1, z_2 \in [-M,M],M>0
 \end{equation}
 where $\mK_6:\R \rightarrow [0,\infty)$ is continuous.

\item \label{C2} $\beta(\vb{x},u)=au+r(\vb{x})$ with $a>0$ and $r\in \BV(\R^2).$
\end{enumerate}
\begin{remark}\normalfont \label{beta} $\beta(\cdot,\cdot)$ satisfies the following properties which will be useful in the sequel.
\begin{enumerate}[label=\roman*.]
\item $\abs{\beta(\vb{x}_1,u)-\beta(\vb{x}_2,u)} =\abs{r(\vb{x}_1)-r(\vb{x}_2)}.$
\item $\abs{\beta(\vb{x},u)-\beta(\vb{x},v)} = a\abs{u-v}.$
\item $\beta^{-1}(\vb{x},u)=\frac{u}{a}-\frac{r(\vb{x})}{a}.$
\item $\abs{\beta^{-1}(\vb{x},u)-\beta^{-1}(\vb{x},v)}=\frac{1}{a}\abs{u-v}.$
\item $\abs{\beta^{-1}(\vb{x}_1,u)-\beta^{-1}(\vb{x}_2,u)}=\frac{1}{a}\abs{r(\vb{x}_1-\vb{x}_2)}.$
\end{enumerate}
\end{remark}
For $\D x,\D y>0,$ consider equidistant spatial grid points $x_i:=i\D x$  and  $y_j:=j\D y$ for $i, j\in\Z.$  
For $\D t>,0$ consider the equidistant temporal grid points $t^n:=n\D t$  and $t^{n+1/2}:=\left(n+1/2 \right)\D t$ for integers $0 \le n\le N$, where $T \in [t^N,t^{N+1})$. 
Let $\la_x:=\D t/\D x$ and $\la_y:=\D t/\D y$. As earlier, let $\chi_{i}(x)$ denote the indicator function of $C_i=[x_i - \D x /2, x_i + \D x /2)$, $\chi_{j}(y)$ denote the indicator function of $C_j=[y_j - \D y /2, y_j + \D y /2) $ and   $\chi_{ij}(x,y)$ denote the indicator function of $C_{ij}:=C_i\times C_j$. Let
$\chi^n(t)$ and $\chi^{n+1/2}(t)$ denote the indicator function of $C^n:=[t^n,t^{n+1/2}),C^{n+1/2}:=[t^{n+1/2},t^{n+1})$ respectively. 
 Given $\D x, \D y >0,$ the total variation of a double sequence $\{a_{ij}\}_{i,j\in \mathbb{Z}}$ is given by

 \begin{eqnarray*}
 \TV_{\D x,\D y}(a):= \D y\sum_{i,j\in \Z} \abs{a_{ij}-a_{i-1,j}}+\D x\sum_{i,j\in \Z} \abs{a_{ij}-a_{i,j-1}}.
 \end{eqnarray*}

Now we define constant approximations, which will be useful in the sequel:
\begin{eqnarray*}
u^{\D}_0(x,y)&:=&\sum\limits_{i,j} \chi_{ij}(x,y)u^0_{ij}\quad \mbox{where }u^0_{ij}=u_0(x_i, y_j)\mbox{ for }i,j\in\Z,\\
r^{\D}(x,y)&:=&\sum\limits_{i,j} \chi_{ij}(x,y)r_{ij}\quad \mbox{where }r_{ij}=r(x_i, y_j)\mbox{ for }i,j\in\Z,\\
k^{\D}_{\alpha}(x,y)&:=&\sum\limits_{i,j} \chi_{ij}(x,y)\frac{\alpha-r_{ij}}{a},\\
\beta^{\D}(x,y,u)&:=&au+r^{\D}(x,y).
\end{eqnarray*}
The  approximations generated by the scheme are denoted by $u_{ij}^n$, where $u_{ij}^n \approx u(x_i, y_j,t^n)$.
The grid function $\{u_{ij}^n\}$ is extended to a function defined on $\Pi_T=\R^2 \times [0,T]$ via
\begin{equation}\label{def_u_De_multi}
u^{\D}(x,y,t) =\sum\limits_{i,j} \chi_{ij}(x,y) \left(\sum_{n=0}^N \left[\chi^n(t) u_{ij}^n +\chi^{n+1/2}(t) u_{ij}^{n+1/2}\right] \right) .
\end{equation}
Similarly, we define another grid function $\beta_{ij}^n=\beta(x_i,y_j,u_{ij}^n)\approx \beta(x_i,y_j,u(x_i,y_j,t^n)) ,$ and is extended to a function defined on $\Pi_T$  via 
\begin{equation}\label{def_beta_De_multi}
\beta^{\D}(x,y,u) =\sum_{n=0}^N \sumi \chi_{ij}(x,y) \chi^n(t) \beta_{ij}^n.
\end{equation}
For $i,j \in \mathbb{Z}$ and $n=0,1,2,\ldots,$ define $\beta_{ij}^n = \beta(x_i,y_j,u_{ij}^n)$ and  $\beta_{ij}^{n+1/2} = \beta(x_i,y_j,u_{ij}^{n+1/2}).$

Now the marching formula is given by
\begin{eqnarray}
\label{xdir}u_{ij}^{n+1/2}&=&u_{ij}^n- \lambda_x \Big(\overline{g}_1(\beta_{ij}^n,\beta_{i+1,j}^n)-\overline{g}_1(\beta_{i-1,j}^n,\beta_{ij}^n) \Big),
\\
\label{ydir}
u_{ij}^{n+1}&=&u_{ij}^{n+1/2}- \lambda_y \Big(\overline{g}_2(\beta_{ij}^{n+1/2},\beta_{i,j+1}^{n+1/2})-\overline{g}_2(\beta_{i,j-1}^{n+1/2},\beta_{ij}^{n+1/2}) \Big),
\end{eqnarray}
where for $l=1,2$, $\bar{g}_l$ denotes the Godunov numerical flux associated with $g_l$:
\begin{equation}
\bar{g}_l(p,q) = 
\begin{cases}
\min_{w \in [p,q]} g_l(w), \quad & p\le q,\\
\max_{w \in [q,p]} g_l(w), \quad & p\ge q.
\end{cases} 
\end{equation}
\begin{lemma}The following bounds hold:
\begin{enumerate}[label=\roman*.]
\item $
\overline{\alpha}_+:=\sup_{(x,y)\in \R^2}\beta(x,y,u_0(x,y)) < \infty$  and $
\overline{\alpha}_-:=\inf_{(x,y)\in \R^2}\beta(x,y,u_0(x,y))>-\infty.$

\item There exists $\mM$ such that \begin{eqnarray}
||k_{\overline{\alpha}_{\pm}}||_{L^{\infty}}<\mM.
\end{eqnarray}
\end{enumerate}

\end{lemma}
Let $\ms = \sup_{\abs{u} \le \mM, (x,y)\in \R^2} \abs{\beta(x,y,u)}$, and
define
$
L_{\beta} = a,  L_g = \mK_6(\ms).
$

Hereafter the ratios $\lambda_x=\frac{\D t}{\D x}$ and $\lambda_y=\frac{\D t}{\D y}$ are fixed and satisfy the $\CFL$ condition:
\begin{eqnarray}\label{cfl_multi}\lambda_x L_{g_1} L_{\beta} \leq 1/2 \text{ and } \lambda_y L_{g_2} L_{\beta} \leq 1/2.\end{eqnarray}

\begin{lemma}\label{lemma:A_u_bound}
                          Under the $\CFL$ condition \eqref{cfl_multi},  the Godunov approximations are bounded:
	\begin{equation}\label{u_bounded}
	\abs{u_{ij}^n} \le \mM, \quad i,j \in \Z,  n=0,1,2,\ldots.
	\end{equation}
    \end{lemma}
   \begin{proof}
   Follows on the similar lines of \cite{GJT_2019}.
\end{proof}
\begin{lemma}\label{beta_tvd_multi}
	Under the $\CFL$ condition, \eqref{cfl_multi} the Godunov scheme is $\TVD$ with respect to $\{\beta_{ij}^{n/2}\}$ in the following sense:
  \begin{eqnarray}\TV_{\D x, \D y}(\beta_{ij}^{(n+1)/2}) \leq \TV_{\D x, \D y} (\beta_{ij}^{n/2}), \quad n=0,1,2,\ldots. \end{eqnarray}\end{lemma}
  \begin{proof}
 Since $\beta(x,y,u)=au+r(x,y),$ the marching formula \eqref{xdir}-\eqref{ydir} implies the following marching formula for $\beta$:
 \begin{eqnarray}
\label{x_dir_beta}\beta_{ij}^{n+1/2}&=&\beta_{ij}^n- \lambda_x \Big(\bar{A}_1(u^n_{ij},u^n_{i+1,j},x_i,x_{i+1},y_j)-\bar{A}_1(u^n_{i-1,j},u^n_{ij},x_{i-1},x_{i},y_j) \Big),\\
\label{y_dir_beta}
\beta_{ij}^{n+1}&=&\beta_{ij}^{n+1/2}- \lambda_y \Big(\bar{A}_2(u^{n+1/2}_{ij},u^{n+1/2}_{i,j+1},x_i,y_j,y_{j+1})-\bar{A}_2(u^{n+1/2}_{i,j-1},u^{n+1/2}_{i,j},x_i,y_{j-1},y_{j})\Big),
\end{eqnarray}
where
\begin{eqnarray*}
\bar{A}_1(u^n_{ij},u^n_{i+1,j},x_i,x_{i+1},y_j)&=&\overline{g}_1(\beta_{ij}^n,\beta_{i+1,j}^n),\\
\bar{A}_2(u^{n+1/2}_{ij},u^{n+1/2}_{i,j+1},x_i,y_j,,y_{j+1})&=&\overline{g}_2(\beta_{ij}^{n+1/2},\beta_{i,j+1}^{n+1/2}).
\end{eqnarray*}
Now, the scheme is monotone and conservative with respect to $\beta$ and thus using Crandall-Tartar lemma for every pair $(p,q) \in \mathbb{Z}^2,$ we have the following $L^1$ contractivity :
\begin{eqnarray}\label{beta_contraction_1}
\sum\limits_{i \in \mathbb{Z}} |\beta^{n+1/2}_{ip}-\beta^{n+1/2}_{iq}| &\leq& \sum\limits_{i \in \mathbb{Z}} |\beta^{n}_{ip}-\beta^{n}_{iq}|,\\\label{beta_contraction_2}
\sum\limits_{j \in \mathbb{Z}} |\beta^{n+1}_{pj}-\beta^{n+1}_{qj}| &\leq& \sum\limits_{j \in \mathbb{Z}} |\beta^{n+1/2}_{pj}-\beta^{n+1/2}_{qj}|. 
\end{eqnarray}Using TVD property \eqref{beta_tvd} for the schemes \eqref{x_dir_beta}-\eqref{y_dir_beta}, one has
\begin{eqnarray}\label{beta_tvd_1}
\sum\limits_{i \in \mathbb{Z}} |\beta^{n+1/2}_{ij}-\beta^{n+1/2}_{i-1,j}|&\leq& \sum\limits_{i \in \mathbb{Z}} |\beta^{n}_{ij}-\beta^{n}_{i-1,j}| \quad \text{ for each }   j\in \mathbb{Z},\\
 \label{beta_tvd_2}\sum\limits_{j\in \mathbb{Z}} |\beta^{n+1}_{ij}-\beta^{n+1}_{i,j-1}|&\leq& 
\sum\limits_{j \in \mathbb{Z}} |\beta^{n+1/2}_{ij}-\beta^{n+1/2}_{i,j-1}| \quad \text{ for each }   i\in \mathbb{Z}.
\end{eqnarray}
For odd $n$, using \eqref{beta_contraction_1} and \eqref{beta_tvd_1}, one has,
\begin{eqnarray*}
\TV_{ \D x, \D y}(\beta_{ij}^{n/2})&=&\Delta y\sum\limits_{i,j \in \mathbb{Z}} |\beta^{n/2}_{ij}-\beta^{n/2}_{i-1j}|+\Delta x \sum\limits_{i,j \in \mathbb{Z}} |\beta^{n/2}_{ij}-\beta^{n/2}_{i,j-1}|,\\&
\leq& \Delta y\sum\limits_{i,j \in \mathbb{Z}} |\beta^{(n-1)/2}_{ij}-\beta^{(n-1)/2}_{i-1,j}| +\Delta x\sum\limits_{i,j \in \mathbb{Z}} |\beta^{(n-1)/2}_{ij}-\beta^{(n-1)/2}_{i,j-1}|,\nonumber
\end{eqnarray*}
which implies the lemma when $n$ is odd. Finally, the proof follows using \eqref{beta_contraction_2} and  \eqref{beta_tvd_2} for even $n$.
\end{proof}
\begin{lemma}\label{lemma:TV_u_multi} Under the $\CFL$ condition \eqref{cfl_multi}, the following properties hold:
\begin{enumerate}[label=\roman*.]
\item \label{tvbeta0} If $\TV_{\D x, \D y} (u^0)< \infty,$ then $\TV_{\D x, \D y} (\beta^0)< \infty.$
\item \label{tvu_multi} Total variation bound on $\{u_{ij}^n\}$: For some $\Delta$-independent constant $\mK_6>0$,
\begin{equation}\label{bv_u_n}
\Delta x \sum\limits_{i,j \in \mathbb{Z}} \abs{u_{i+1,j}^{n} - u_{ij}^{n}}+\Delta y \sum\limits_{i,j \in \mathbb{Z}} \abs{u_{i,j+1}^{n} - u_{ij}^{n}} \le \mK_6, \quad n=0,1,2,\ldots.
\end{equation}
\item Discrete time continuity estimates: 
\begin{eqnarray}\label{time_cont_multi}  \sum\limits_{i,j\in \mathbb{Z}} \abs{u_{ij}^{n+1}-u_{ij}^n} \leq \mK_7, \quad n=0,1,2,\ldots.\end{eqnarray}
\item Discrete entropy inequalities: 
\begin{eqnarray}\label{dis_ad_multi_1}
\abs{u^{\nph}_{ij}- k^{\alpha}_{ij}} 
	&\le& \abs{u_{ij}^{n} - k^{\alpha}_{ij}}
	- \lambda_x (\mP^{n}_{\iph,j} - \mP^n_{\imh, j}),\mbox{ for all }i,j\in\Z,\\\label{dis_ad_multi_2}
	\abs{u^{n+1}_{ij}- k^{\alpha}_{ij}} 
	&\le& \abs{u_{ij}^{\nph} - k^{\alpha}_{ij}}
	- \lambda_y (\mQ^{\nph}_{i,\jph} - \mQ^{\nph}_{i, \jmh}),\mbox{ for all }i,j\in\Z,
\end{eqnarray}
where
\begin{eqnarray*}
\mP^n_{\iph,j}&=&\bar{A}_1(u^n_{ij}\vee k^{\alpha}_{ij},u^n_{i+1,j}\vee k^{\alpha}_{i+1,j},x_i,x_{i+1},y_j)\\
&&\quad \quad \quad-\bar{A}(u^n_{ij}\wedge k^{\alpha}_{ij},u^n_{i+1,j}\wedge k^{\alpha}_{i+1,j},x_i,x_{i+1},y_j),\\
\mQ^{\nph}_{i,\jph}&=&\bar{A}_2(u^{\nph}_{ij}\vee k^{\alpha}_{ij},u^{\nph}_{i,j+1}\vee k^{\alpha}_{i,j+1},x_i,y_j,y_{j+1})\\&& \quad \quad\quad -\bar{A}(u^{\nph}_{ij}\wedge k^{\alpha}_{ij},u^{\nph}_{i,j+1}\wedge  k^{\alpha}_{i,j+1},x_{i},y_j,y_{j+1}).
\end{eqnarray*}
\end{enumerate}
\end{lemma}
\begin{proof}
We have
\begin{eqnarray}
\abs{\beta_{ij}^0-\beta_{i,j-1}^0}\leq a\abs{u_{ij}^0-u_{i,j-1}^0} +\abs{ r_{ij}-r_{i,j-1}},\\
\abs{\beta_{ij}^0-\beta_{i-1,j}^0}\leq a\abs{u_{ij}^0-u_{i-1,j}^0} +\abs{ r_{ij}-r_{i-1,j}}.
\end{eqnarray}
Thus for $\TV_{\D x, \D y} (u^0) < \infty $ and $ \TV_{\D x, \D y}(r) < \infty,$  Lemma \ref{beta_tvd_multi} implies, 
\begin{eqnarray*}
\TV _{\D x, \D y}(\beta^{n/2})&=&\Delta y\sum\limits_{i,j \in \mathbb{Z}} |\beta^{n/2}_{ij}-\beta^{n/2}_{i-1j}|+\Delta x \sum\limits_{i,j \in \mathbb{Z}} |\beta^{n/2}_{ij}-\beta^{n/2}_{i,j-1}| \\&\leq& a \TV_{\D x,\D y} TV(u^0)+ \TV_{\D x, \D y} (r).
\end{eqnarray*}
This proves \eqref{tvbeta0}.\\
Consider,
\begin{eqnarray*}
\abs{u_{ij}^{n/2}-u_{i,j-1}^{n/2}} \leq \frac{1}{a}\left[\abs{\beta_{ij}^{n/2}-\beta_{i,j-1}^{n/2}} + \abs{r_{ij}-r_{i,j-1}}  \right],\\
\abs{u_{ij}^{n/2}-u_{i-1,j}^{n/2}} \leq \frac{1}{a} \left[\abs{\beta_{ij}^{n/2}-\beta_{i-1,j}^{n/2}} + \abs{r_{ij}-r_{i,j-1}}  \right].
\end{eqnarray*}
Thus, 
\begin{eqnarray}\TV _{\D x, \D y} (u^{n/2}) \leq  \frac{1}{a} \TV_{\D x, \D y}(\beta) +\TV_{\D x, \D y} (r).
\end{eqnarray} Thus \eqref{tvu_multi} follows.
The proof of \eqref{time_cont_multi} follows from \eqref{time_cont}. The proof of the discrete entropy inequalities
\eqref{dis_ad_multi_1}-\eqref{dis_ad_multi_2} can be obtained using \eqref{ent_discrete_A}.
\end{proof}

\begin{theorem}\label{theorem:conv_multi}
	Assume that the flux function $\vb{A}(\vb{x},u)=\vb{g}(\beta(\vb{x},u))$ satisfies   Assumptions \descref{C1}{C-1} and \descref{C2}{C-2}, 
	and that $u_0\in \BV(\R^d)$. 

	Then as the mesh size $\D \rightarrow 0$, the approximations $u^{\D}$ generated by the Godunov scheme described 
	above converge in $L^1_{\loc}(Q)$ and pointwise a.e. in $Q$ to the unique adapted entropy solution 
	$u \in L^{\infty}(Q) \cap C([0,T]:L^1_{\loc}(\R^d))$ corresponding to the Cauchy problem \eqref{eq:discont}, 
	\eqref{eq:data} with initial data $u_0$. In addition, the total variation $u(\cdot,t)$ is uniformly  bounded for $t\geq 0$.
\end{theorem}
\begin{proof}
From  the spatial variation bound on $\{u_{ij}^n \}$ and the time continuity estimate obtained in Lemma~\ref{lemma:TV_u_multi},
 we have convergence of the
approximations $u^{\D}$ along a subsequence in $L^1_{\textrm{loc}}(Q)$ and boundedly a.e. to 
some $u \in L^{\infty}(Q) \cap C([0,T]:L^1_{\loc}(\R^d))$.
Since the scheme satisfies the discrete adapted entropy inequality \eqref{dis_ad_multi_1}-\eqref{dis_ad_multi_2}, we can invoke the dimensional splitting arguments of Crandal-Majda \cite{CM_1980}, in the adapted entropy set up to show that the limit indeed satisfies the adapted entropy condition. 

By Lemma~\ref{lemma:TV_u_multi},
we have a spatial variation bound on $u^{\D}(\cdot,t)$ which is independent of the mesh size, i.e.,
	for some $\mK_6>0$ independent of the mesh size $\D$,
	\begin{equation}\label{u_bv_bound}
	\TV(u^{\D}(\cdot,t))  \le \mK_6.
	\end{equation}
	Since $\TV(u(\cdot,t)) \le \liminf \TV(u^{\D}(\cdot,t))$, we also have $\TV(u(\cdot,t)) \le \mK_6.$
\end{proof}
\section{Error Estimates} \label{sec_error_est}
In this section, we estimate the rate of convergence of the numerical methods introduced in the previous section. The idea is to prove the Kuznetsov type lemma based on the adapted entropy formulation. We begin by listing some of the technical tools required to prove the Kuznetsov lemma. We assume that $u_0,r \in \BV(\R^d) \cap L^1(\R^d)$ and  the fluxes satisfy the assumptions detailed in the previous section.

  \begin{definition}
 Let $\Pi_T=\R^d \times [0,T].$ We define $\Phi^{\eta,\epsilon}: {\Pi_T}^2 \rightarrow \R$ by,  \begin{eqnarray*}
\Phi^{\eta,\epsilon}(t,\vb{x},s,\vb{y})=\omega_{\epsilon}(t-s)\omega_{\eta}(\vb{x}-\vb{y}),
\end{eqnarray*}
where  for $\vb{z}\in \R^d,$  $\omega_{\eta}(\vb{z}):=\frac{1}{{\eta}^d}\prod_{i=1}^{d}\omega\Big(\frac{z_i}{\eta}\Big)$ is a mollifier such that $\omega\in C^{\infty}(\R;\R)$ is an even function and satisfies the following:
 \begin{eqnarray}\spt(w) \subset [0,1], \quad  0\leq \omega(z)\leq 1 \, \text{ and } 
 \int\limits_{\R}w(z)d z=1.
\end{eqnarray}
\end{definition}
For further calculations, we note the following properties of $\Phi^{\eta,\epsilon}$:
\begin{enumerate}
    \item \begin{eqnarray}\label{S1} \Phi^{\eta,\epsilon}_{x_i}&=&\frac{\partial}{\partial x_i}{\Phi^{\eta,\epsilon}}(t,\vb{x},s,\vb{y})=\omega_{\epsilon}(t-s) \omega^{'}_{\eta}(x_i-y_i)\prod_{j \neq i}\omega_{\eta}(y_j-y_j) \\\nonumber&& \quad \quad \quad \quad \quad \quad \quad \quad\quad \quad \quad=-\frac{\partial}{\partial y_i}\Phi^{\eta,\epsilon}(t,\vb{x},s,\vb{y})=-\Phi^{\eta,\epsilon}_{y_i}.\end{eqnarray}
    \item \begin{eqnarray}\label{S2}\Phi^{\eta,\epsilon}_t=\frac{\partial}{\partial t}\Phi^{\eta,\epsilon}(t,\vb{x},s,\vb{y})=\omega^{'}_{\epsilon}(t-s)\omega_{\eta}( \vb{x}-\vb{y})=- \frac{\partial}{\partial s}{\Phi^{\eta,\epsilon}}(t,\vb{x},s,\vb{y})=-\Phi^{\eta,\epsilon}_s.\end{eqnarray}
\item 
\begin{equation}
\begin{aligned}[t]\label{S3}\Phi^{\eta,\epsilon}(t,\vb{x},s,\vb{y})=\Phi^{\eta,\epsilon}(t,\vb{y},s,\vb{x})=\Phi^{\eta,\epsilon}(s,\vb{x},t,\vb{y})=\Phi^{\eta,\epsilon}(s,\vb{y},t,\vb{x}).\end{aligned}\end{equation}
\item \begin{eqnarray}\label{S4}\int\limits_{\R^d}w_{\eta}( \vb{x}-\vb{y})d \vb{y}=1 \text{ and } \int\limits_0^T w_{\epsilon}(t-s)ds \leq 1, \quad \text{ for all }\vb{x}\in \R^d, t\geq 0, \end{eqnarray}
\item There exists $C$ independent of $\eta$ and $\epsilon$ such that, \begin{eqnarray}\label{S5}\int\limits_{\R^d}| \partial_{x_i}w_{\eta}(\vb{x}-\vb{y})|d\vb{y}\leq \frac{C}{\eta} \text{ and } \int\limits_0^T |w'_{\epsilon}|(t-s)ds \leq \frac{C}{\epsilon}, \quad \text{ for all }\vb{x}\in \R^d, t\geq 0. \end{eqnarray}
\end{enumerate}
\begin{definition}For $\sigma >0,$ define the following
\begin{enumerate}[i.]
    \item $\kappa:=\{u:\Pi_T\rightarrow \R: ||u(\cdot, t)||_{L^{\infty}}\le  k, |u(\cdot,t)|_{\BV} \leq k\}.$
      
    \item $\nu_t(u,\sigma) :=\sup_{|\tau| \leq \sigma}||u(t+\tau)-u(t)||_1.$
    \item $\nu(u,\sigma):=\sup\limits_{0<t<T}\nu_t(u,\sigma)=\sup\limits_{t\in(0,T)}\sup_{|\tau| \leq \sigma}||u(t+\tau)-u(t)||_1.$
\end{enumerate}
\end{definition}
\textbf{Remark:} If $u_0 \in \BV(\R^d) \cap L^1(\R^d)$ then there exists $L$ such that adapted entropy solution satisfies $\nu(u,\sigma) \leq L\sigma.$

    \begin{definition}\begin{eqnarray}\label{W1}
    \wedge_T(u,\phi,k_{\alpha})&:=&
    \int_{\Pi_T}\Big( |u(t,\vb{x})-k_\alpha(\vb{x})|\phi_{t}+\sum\limits_{i=1}^d\sgn (u(t,\vb{x})-k_\alpha(\vb{x})) \Big(A_i(\vb{x},u(t,\vb{x}))-\alpha\Big)\phi_{x_i}\Big)  d\vb{x} d{t} \nonumber\\
    &&\label{W2}-\int_{\R^d}|u(T,\vb{x})-k_\alpha(\vb{x})|\phi(T,\vb{x}) d\vb{x}+\int_{\R^d}|u_0(\vb{x})-k_\alpha(\vb{x})|\phi(0,\vb{x}) d\vb{x}. \\
   \displaystyle\wedge_{\eta,\epsilon}(u,v)&:=&\int_{\Pi_T}\wedge_T(u(\cdot,\cdot),\phi^{\eta,\epsilon}(\cdot,\cdot,s,\vb{y}), \tilde{v}(s,\vb{y},\vb{x})) d\vb{y} ds.\\ \label{W3}
   \displaystyle\wedge_{\eta,\epsilon}(v,u)&:=&\int_{\Pi_T}\wedge_T(v(\cdot,\cdot),\Phi^{\eta,\epsilon}(t,\vb{x},\cdot,\cdot), \tilde{u}(t,\vb{x},\vb{y})) d\vb{x} dt. 
    \end{eqnarray}\end{definition}

\begin{lemma}\label{kuznetsov_adapted}
Let $v$ be the solution of IVP \eqref{eq:discont}-\eqref{eq:data}  and $u \in \kappa. $ For $0<\epsilon<T$ and $\eta >0,$ then
\begin{eqnarray}\label{kza}
 \norma{u(\cdot,T)-v(\cdot,T)}_{L^1(\R^d)} &\leq&  \norma{u_0-v_0}_{L^1(\R^d)}+ C \Big[L\epsilon +\TV(r)|\eta|+ \TV(v)|\eta|\\ &&  \,\,\,\,+ \nu(u, \epsilon)\Big]-\wedge_{\eta,\epsilon}(u,v).\nonumber
\end{eqnarray}
where $C$ is independent of the mesh size $\D.$
\end{lemma}
\begin{proof}

  Adding $\displaystyle\wedge_{\eta,\epsilon}(v,u)$ and $\displaystyle\wedge_{\eta,\epsilon}(u,v),$ we get the following
\begin{eqnarray*}
&&{\color{black}\displaystyle\wedge_{\eta,\epsilon}(v,u)}+ {\color{black}\displaystyle\wedge_{\eta,\epsilon}(u,v)}\\
&&\int_{\Pi_T}\Big( |u(t,\vb{x})-\tilde{v}(\vb{y},s,\vb{x})|\Phi^{\eta,\epsilon}_{t} d\vb{x} d{t} d\vb{y} ds\\
&+&\int_{\Pi_T}\left[\sum\limits_{i=1}^d\left( \sgn (u(t,\vb{x})-\tilde{v}(\vb{y},s,\vb{x})) (A_i(\vb{x},u(\vb{x},t))-A_i(\vb{y},v(\vb{y},s)))\right) \right]\Phi^{\eta,\epsilon}_{x_i} d\vb{x} dt d\vb{y} ds\\
    &-&\int_{\Pi_T}\int_{\R^d}|u(T,\vb{x})-\tilde{v}(s,\vb{y},\vb{x})|\Phi^{\eta,\epsilon}(\vb{x},T,\vb{y},s)d\vb{x} d\vb{y} ds\\
    &+&\int_{\Pi_T}\int_{\R^d}|u_0(\vb{x})-\tilde{v}(s,\vb{y},\vb{x})|\Phi^{\eta,\epsilon}(\vb{x},0,\vb{y},s)d\vb{x} d\vb{y} ds+\int_{\Pi_T}\Big( |v(s,\vb{y})-\tilde{u}(t,\vb{x},\vb{y})|\Phi^{\eta,\epsilon}_{s} d\vb{y} ds d\vb{x} dt\\
    &+&\int_{\Pi_T}
    \left[ \sum\limits_{i=1}^d\sgn (v(s,\vb{y})-\tilde{u}(t,\vb{x},\vb{y})) (A_i(\vb{y},v(s,\vb{y}))-A_i(\vb{x},u(t,\vb{x})))\right]\Phi^{\eta,\epsilon}_{y_i}d\vb{y} d{s} d\vb{x} dt\\
    &-&\int_{\Pi_T}\int_{\R^d}|v(T,\vb{y})-\tilde{u}(t,\vb{x},\vb{y})|\Phi^{\eta,\epsilon}(\vb{x},t,\vb{y},T)d\vb{y}  d\vb{x} dt\\
   & +&\int_{\Pi_T}\int_{\R^d}|v_0(\vb{y})-\tilde{u}(t,\vb{x},\vb{y})|\Phi^{\eta,\epsilon}(\vb{x},t,\vb{y},0)d\vb{y}  d\vb{x} dt.
\end{eqnarray*}
From \eqref{S1}, terms involving $\Phi^{\eta,\epsilon}_{x_i}$ and $\Phi^{\eta,\epsilon}_{y_i}$ cancel each other. Now invoking symmetry of $\Phi^{\eta,\epsilon}$ given by \eqref{S1}--\eqref{S3}, we have the following
$$\wedge_{\eta,\epsilon}(u,v)=-\wedge_{\eta,\epsilon}(v,u)-\mathcal{A}+\mathcal{B}+\mathcal{C},$$ where,
 \begin{eqnarray*}
 \mathcal{A} &=&\int_{\Pi_T}\int_{\R^d}\Big(|u(T,\vb{x})-\tilde{v}(s,\vb{y},\vb{x})|+|v(T,\vb{y})-\tilde{u}(t,\vb{x},\vb{y})|\Big)\Phi^{\eta,\epsilon}(\vb{x},s,\vb{y},T) d\vb{y} d\vb{x} ds\\ 
 &=&\int_{0}^T w_{\epsilon}(T-s)\int_{\R^{2d}}\Big(|u(t,\vb{x})-\tilde{v}(s,\vb{y},\vb{x})|+|v(T,\vb{y})-\tilde{u}(t,\vb{x},\vb{y})|\Big)w_{\eta}( \vb{x}-\vb{y})d\vb{y}  d\vb{x} ds.\\ \mathcal{B} &=&\int_{\Pi_T}\int_{\R^d}\Big(|u_0(\vb{x})-\tilde{v}(s,\vb{y},\vb{x})|+|v_0(\vb{y})-\tilde{u}(t,\vb{x},\vb{y})|\Big)\Phi^{\eta,\epsilon}(\vb{x},s,\vb{y},0) d\vb{x}  d\vb{y} ds\\ &=&\int_{0}^T w_{\epsilon}(T-s)\int_{\R^{2d}}\Big(|u_0(\vb{x})-\tilde{v}(s,\vb{y},\vb{x})|+|v_0(\vb{y})-\tilde{u}(t,\vb{x},\vb{y})|\Big)w_{\eta}( \vb{x}-\vb{y}) d\vb{x}  d\vb{y} ds. \\
\mathcal{C} &=&\int\limits_{{\Pi_T}^2}\Big( |u(t,\vb{x})-\tilde{v}(s,\vb{y},\vb{x})|-|v(s,\vb{y})-\tilde{u}(t,\vb{x},\vb{y})|\Big) w'_{\epsilon}(t-s) w_{\eta}( \vb{x}-\vb{y}) d\vb{x} d\vb{y} ds dt.
 \end{eqnarray*}
 since $v$ is the solution, $\wedge_{\eta,\epsilon}(v,u)\ge0,$ implying that \begin{eqnarray}\label{ABC}\mathcal{A}\le \mathcal{B} + \mathcal{C} -\wedge_{\eta,\epsilon}(u,v).\end{eqnarray}
\begin{enumerate}[label=\textbf{Claim} \Roman*]
\item We have the following lower bound on $\mathcal{A}:$

\begin{eqnarray}\label{A}
\mathcal{A} \geq \norma{u(\cdot,T)-v(\cdot,T)}_{L^1(\R^d)}- C \Big( L\epsilon + \TV(r)|\eta|+\TV(v) |\eta|+\nu(u, \epsilon) \Big).
\end{eqnarray}
\newline
To prove the claim we make the following estimates.
 \begin{enumerate}
\item Estimation of $|u(T,\vb{x})-\tilde{v}(s,\vb{y},\vb{x})|$:
 \\
 Consider,
\begin{eqnarray*}
|u(T,\vb{x})-v(T,\vb{x})|&=&|u(T,\vb{x})-\tilde{v}(s,\vb{y},\vb{x})+\tilde{v}(s,\vb{y},\vb{x})-\tilde{v}(T,\vb{y},\vb{x})+\tilde{v}(T,\vb{y},\vb{x})-v(T,\vb{x})|\\
&\leq& |u(T,\vb{x})-\tilde{v}(s,\vb{y},\vb{x})|+|\tilde{v}(s,\vb{y},\vb{x})-\tilde{v}(T,\vb{y},\vb{x})|+|\tilde{v}(T,\vb{y},\vb{x})-v(T,\vb{x})|.
\end{eqnarray*}
Thus we have,
\begin{eqnarray*}
|u(T,\vb{x})-\tilde{v}(s,\vb{y},\vb{x})| &\geq& |u(T,\vb{x})-v(T,\vb{x})|-|\tilde{v}(s,\vb{y},\vb{x})-\tilde{v}(T,\vb{y},\vb{x})|-|\tilde{v}(T,\vb{y},\vb{x})-v(T,\vb{x})|.
\end{eqnarray*}Using the definition of $\tilde{v}$ we get,
\begin{eqnarray*}
|\tilde{v}(T,\vb{y},\vb{x})-\tilde{v}(s,\vb{y},\vb{x})|&=&|\beta^{-1}(\vb{x},\beta(\vb{y},v(T,\vb{y})))-\beta^{-1}(\vb{x},\beta(\vb{y},v(s,\vb{y})))|
\\&\leq& \frac{1}{a}|\beta(\vb{y},v(T,\vb{y}))-\beta(\vb{y},v(s,\vb{y}))|\\&\leq & |v(T,\vb{y})-v(s,\vb{y})|.
\end{eqnarray*} 
Invoking the properties of $\beta,$ we get the following\begin{eqnarray*}
|\tilde{v}(T,\vb{y},\vb{x})-\tilde{v}(T,\vb{x},\vb{x})|&=& |\beta^{-1}(\vb{x},\beta(\vb{y},v(T,\vb{y})))-\beta^{-1}(\vb{x},\beta(\vb{x},v(T,\vb{x})))|\\
&\leq & \frac{1}{a}|\beta(\vb{y},v(T,\vb{y}))-\beta(\vb{x},v(T,\vb{x}))|\\
&=& \frac{1}{a}|\beta(\vb{y},v(T,\vb{y}))-\beta(\vb{x},v(T,\vb{y}))+\beta(\vb{x},v(T,\vb{y}))-\beta(\vb{x},v(T,\vb{x}))|.\\
&\leq & \frac{1}{a}|r(\vb{x})-r(\vb{y})|+ |v(T,\vb{y})-v(T,\vb{x})|.
\end{eqnarray*} Combining all these estimates we get,
\begin{eqnarray}\label{a} |u(T,\vb{x})-\tilde{v}(s,\vb{y},\vb{x})| &\geq& |u(T,\vb{x})-v(T,\vb{x})|- |v(T,\vb{y})-v(s,\vb{y})|\nonumber\\&& \quad \quad -\left[\frac{1}{a}|r(\vb{x})-r(\vb{y})|+ |v(T,\vb{y})-v(T,\vb{x})|\right].\end{eqnarray}

\item Estimation of  $|v(T,\vb{y})-\tilde{u}(t,\vb{x},\vb{y})|$:\\
Consider $|u(T,\vb{x})-v(T,\vb{x})|$, add and subtract $\tilde{u}(s,\vb{x},\vb{y})$ and $v(T,\vb{y})=\tilde{v}(T,\vb{y},\vb{y})$ to get,
\begin{eqnarray*}
|u(T,\vb{x})-v(T,\vb{x})|&=&|u(T,\vb{x})-\tilde{u}(s,\vb{x},\vb{y})+\tilde{u}(s,\vb{x},\vb{y})-v(T,\vb{y})+v(T,\vb{y})-v(T,\vb{x})|\\
&\leq &|u(T,\vb{x})-\tilde{u}(s,\vb{x},\vb{y})|+|\tilde{u}(s,\vb{x},\vb{y})-v(T,\vb{y})|+|v(T,\vb{y})-v(T,\vb{x})|.
\end{eqnarray*}
Thus we have,
\begin{eqnarray*}
 |\tilde{u}(s,\vb{x},\vb{y})-v(T,\vb{y})| \geq |u(T,\vb{x})-v(T,\vb{x})|- |u(T,\vb{x})-\tilde{u}(s,\vb{x},\vb{y})| - |v(T,\vb{y})-v(T,\vb{x})|.
\end{eqnarray*}
\begin{eqnarray*}
|u(T,\vb{x})-\tilde{u}(s,\vb{x},\vb{y})|&=&|\beta^{-1}(\vb{x},\beta(\vb{x},u(T,\vb{x})))-\beta^{-1}(\vb{y},\beta(\vb{x},u(s,\vb{x})))|\\
&\leq& \frac{1}{a}|r(\vb{x})-r(\vb{y})|+ |u(T,\vb{x})-u(s,\vb{x})|.
\end{eqnarray*}
Combining all these estimates we get, \begin{eqnarray}\label{b}|v(T,\vb{y})-\tilde{u}(t,\vb{x},\vb{y})|&\geq& |u(T,\vb{x})-v(T,\vb{x})| - |v(T,\vb{y})-v(T,\vb{x})|\nonumber\\&&\quad \quad \quad-\left[\frac{1}{a}|r(\vb{x})-r(\vb{y})|+ |u(T,\vb{x})-u(s,\vb{x})|\right].\end{eqnarray}
\end{enumerate}

Adding \eqref{a} and \eqref{b}, for some $C>0$ we get the following estimate:
\begin{eqnarray*}
|u(T,\vb{x})-\tilde{v}(s,\vb{y},\vb{x})|+|\tilde{u}(s,\vb{x},\vb{y})-v(T,\vb{y})|
&\geq& 2|u(T,\vb{x})-v(T,\vb{x})|- C |v(T,\vb{y})-v(s,\vb{y})|\\&-& C \left[|r(\vb{x})-r(\vb{y})|+ |v(T,\vb{y})-v(T,\vb{x})|\right] \\ &-&C \left[ |r(\vb{x})-r(\vb{y})|+|u(T,\vb{x})-u(s,\vb{x})|\right] \\
&-& |v(T,\vb{y})-v(T,\vb{x})|\\
&\geq & 2|u(T,\vb{x})-v(T,\vb{x})|-C \Big[ |v(T,\vb{y})-v(s,\vb{y})|\\&+&|r(\vb{x})-r(\vb{y})|+ |v(T,\vb{y})-v(T,\vb{x})|\\&+&|u(T,\vb{x})-u(s,\vb{x})|\Big].
\end{eqnarray*}

Thus \begin{eqnarray}\label{estA}
\nonumber\mathcal{A}&=&\int_{0}^T w_{\epsilon}(T-s)\int_{\R^{2d}}\Big(|u(t,\vb{x})-\tilde{v}(s,\vb{y},\vb{x})|+|v(T,\vb{y})-\tilde{v}(t,\vb{x},\vb{y})|\Big)w_{\eta}( \vb{x}-\vb{y})d\vb{y}  d\vb{x} ds \\\nonumber
&\geq&  \int_{0}^T w_{\epsilon}(T-s)\int_{\R^{2d}} 2|u(T,\vb{x})-v(T,\vb{x})|-C\Big(|v(T,\vb{y})-v(s,\vb{y})|+|r(\vb{x})-r(\vb{y})|\\&&\quad \quad  +|v(T,\vb{y})-v(T,\vb{x})|+|u(T,\vb{x})-u(s,\vb{x})|\Big)w_{\eta}( \vb{x}-\vb{y})d\vb{y}  d\vb{x} ds.
\end{eqnarray}

To obtain the desired lower bound on $\mathcal{A},$ we estimate each term on the right side of \eqref{estA} as follows:
\begin{enumerate}[label=\textbf{Term} \roman*]
\item $\int_{0}^T \left[w_{\epsilon}(T-s)\int_{\R^{2d}}\Big( |u(T,\vb{x})-v(T,\vb{x})|\Big)w_{\eta}( \vb{x}-\vb{y})d\vb{y}  d\vb{x} \right] ds. $ \\ \newline By symmetry of $w$ we have \begin{eqnarray*}\int_{0}^T\omega_{\epsilon}(T-s)ds=\int_{0}^T\omega_{\epsilon}(s)ds=\frac{1}{2},\end{eqnarray*} Now applying Fubini-Tonellis's theorem we get,
\begin{eqnarray*}\int_{0}^T \left[w_{\epsilon}(T-s)\int_{\R^{2d}}\Big( |u(T,\vb{x})-v(T,\vb{x})|\Big)w_{\eta}( \vb{x}-\vb{y})d\vb{y}  d\vb{x} \right] ds = \frac{1}{2} \norma{u(T,\cdot)-v(T,\cdot)}_{L^1(\R^d)}.
\end{eqnarray*}
\item $\int_{0}^T w_{\epsilon}(T-s)\int_{\R^{2d}}\Big( |v(T,\vb{y})-v(s,\vb{y})|\Big)w_{\eta}( \vb{x}-\vb{y})d\vb{y}  d\vb{x} ds.  $\\ \newline   Since the support of $w_{\epsilon} \subset[-\epsilon,\epsilon],$ using the time continuity of $v$ we get,\begin{eqnarray*}\int_{0}^T w_{\epsilon}(T-s)\int_{\R^{2d}}\Big( |v(T,\vb{y})-v(s,\vb{y})|\Big)w_{\eta}( \vb{x}-\vb{y})d\vb{y}  d\vb{x} ds \leq \frac{1}{2} L \epsilon.
\end{eqnarray*}

\item $\int_{0}^T w_{\epsilon}(T-s)\int_{\R^{2d}}\Big( |r(\vb{x})-r(\vb{y})|\Big)w_{\eta}( \vb{x}-\vb{y})d\vb{y}  d\vb{x} ds. $\\ \newline 
Note that,\begin{eqnarray*}
   \int_{\R^{2d}}\omega_{\eta}( \vb{x}-\vb{y})|r(\vb{x})-r(\vb{y})| d\vb{x}d\vb{y}
      \leq |\eta| \TV(r),
     \end{eqnarray*}and thus we have,
\begin{eqnarray*}\int_{0}^T w_{\epsilon}(T-s)\int_{\R^{2d}}\Big( |r(\vb{x})-r(\vb{y})|\Big)w_{\eta}( \vb{x}-\vb{y})d\vb{y}  d\vb{x} ds \leq \frac{1}{2} \TV(r)|\eta|.
\end{eqnarray*}
\item $\int_{0}^T w_{\epsilon}(T-s)\int_{\R^{2d}}\Big( |v(T,\vb{x})-v(T,\vb{y})|\Big)w_{\eta}( \vb{x}-\vb{y})d\vb{y}  d\vb{x} ds.$ \\ \newline Since $v(T,\cdot)$ has bounded variation, repeating the arguments as in the previous step, we get,  \begin{eqnarray*}
\int_{0}^T w_{\epsilon}(T-s)\int_{\R^{2d}}\Big(|v(T,\vb{y})-v(T,\vb{x})| \Big)w_{\eta}( \vb{x}-\vb{y})d\vb{y}  d\vb{x} ds  \leq  \frac{1}{2} \TV(v)||\eta|.
\end{eqnarray*}
\item $\int_{0}^T w_{\epsilon}(T-s)\int_{\R^{2d}}\Big(|u(T,\vb{x})-u(s,\vb{x})| \Big)w_{\eta}( \vb{x}-\vb{y})d\vb{y}  d\vb{x} ds. $ \\ \newline Note that $w_{\epsilon}(T-s)$ is zero for  $T-s > \epsilon.$ Thus invoking the definition of $\nu(u,\epsilon)$ we get,\begin{eqnarray*}\int_{0}^T w_{\epsilon}(T-s)\int_{\R^{2d}}\Big(|u(T,\vb{x})-u(s,\vb{x})| \Big)w_{\eta}( \vb{x}-\vb{y})d\vb{y}  d\vb{x} ds \leq \frac{1}{2}\nu(u, \epsilon).
\end{eqnarray*}
Combining all these estimates, we get the desired lower bound on $\mathcal{A}$.
\end{enumerate}
\item We have the following upper bound on $\mathcal{B}.$\begin{eqnarray} \label{B}
\mathcal{B} \leq  \norma{u(\cdot,0)-v(\cdot,0)}_{L^1(\R^d)}+  C \Big( L\epsilon + \TV(r)|\eta|+\TV(v) |\eta|+\nu(u, \epsilon) \Big).
\end{eqnarray}
Claim follows by repeating the arguments done in the estimation of $\mathcal{A},$
 for $|u_0(\vb{x})-\tilde{v}(s,\vb{y},\vb{x})|+|v_0(\vb{y})-\tilde{v}(t,\vb{x},\vb{y})|.$ 

\item\begin{eqnarray}\label{C}\mathcal{C}=0.\end{eqnarray}

By the definition of $\tilde{u}$ and $\tilde{v}$ we have,
\begin{eqnarray*}
\beta(\vb{x},u(t,\vb{x}))=a u(t,\vb{x})+r(\vb{x})=a \tilde{u}(t,\vb{x},\vb{y})+r(\vb{y})=\beta(\vb{y},\tilde{u}(t,\vb{x},\vb{y})),\\
\beta(\vb{x},\tilde{v}(s,\vb{y},\vb{x}))=a \tilde{v}(s,\vb{y},\vb{x})+r(\vb{x})=a v(s,\vb{y})+r(\vb{y})=\beta(\vb{y},v(s,\vb{y})).
\end{eqnarray*}

Which implies
\begin{eqnarray*}
u(t,\vb{x})-\tilde{v}(s,\vb{y},\vb{x})=\tilde{u}(t,\vb{x},\vb{y})-v(s,\vb{y}),
\end{eqnarray*}and hence
\begin{eqnarray*}
\abs{u(t,\vb{x})-\tilde{v}(s,\vb{y},\vb{x})}=\abs{\tilde{u}(t,\vb{x},\vb{y})-v(s,\vb{y})}.
\end{eqnarray*}
Thus we have $\mathcal{C}=0$ and claim is proved.

\end{enumerate}

Substituting the values of \eqref{A}-\eqref{C}  in \eqref{ABC}, we have
\begin{eqnarray*}
\norma{u(T,\cdot)-v(T,\cdot)}_{L^1(\R^d)} \leq  \norma{u_0-v_0}_{L^1(\R^d)} + C \Big[L\epsilon +\TV(r)|\eta|+\TV(v)|\eta|+ \nu(u, \epsilon)\Big]-\wedge_{\eta,\epsilon}(u,v).
\end{eqnarray*}
which completes the proof of the lemma.
\end{proof}
\begin{remark}\label{kuz}The terms involving $\TV(r)$ are absent in the original Kuznetsov lemma where the flux is homogeneous.
\end{remark}
Before moving on to the proof of the error estimate, we introduce the following notations:
\begin{eqnarray*}
\eta_{ij}^{n/2}&:=&\abs{u_{ij}^{n/2}-k^{ij}_{\alpha}},\\
p_{ij}^{n/2}&:=&\sgn(u^{n/2}_{ij}-k_{\alpha}^{ij}) 
\big(A_1(x_i,y_j,u^{n/2}_{ij})-A_1(x_i,y_j,k_{\alpha}^{ij})\big) =A_1(x_i,y_j,u^{n/2}_{ij}\lor k_{\alpha}^{ij})-A_1(x_i,y_j,u^{n/2}_{ij} \wedge k_{\alpha}^{ij}),\\[2mm]
q_{ij}^{n/2}&:=&\sgn(u^{n/2}_{ij}-k_{\alpha}^{ij}) \big(A_2(x_i,y_j,u^{n/2}_{ij})-A_2(x_i,y_j,k_{\alpha}^{ij})\big) =A_2(x_i,y_j,u^{n/2}_{ij}\lor k_{\alpha}^{ij})-A_2(x_i,y_j,u^{n/2}_{ij} \wedge k_{\alpha}^{ij}).
\end{eqnarray*}Now we state and prove the convergence rate theorem.
\begin{theorem}\label{CR} [Convergence rate for conservation laws with discontinuous flux]
Let $u$ be the entropy solution of \eqref{eq:discont}-\eqref{eq:data} and $u^{\D}$ the numerical solution given by \eqref{scheme_A}-\eqref{def_bar_A_direct}.
Then we have the following convergence rate:
 \begin{eqnarray*}\norma{u^{\D}(T,\cdot)-v(T,\cdot)}_{L^1(\R^d)} &=& \mathcal{O}(\sqrt{\Delta t}),
 \end{eqnarray*}
for some constant $C$ independent of $\Delta t$.
\end{theorem}
\begin{proof}
Proof is in the spirit of the error estimates for one dimensional  scalar conservation laws with  space independent fluxes presented in \cite{HRbook}. We prove the result for $d=2$ and the proof follows similarly for higher dimensions. Let $\eta=\epsilon=\sqrt{\D t}.$  In view of the previous lemma, it is enough to show the following:
\begin{eqnarray}\label{nu}
\nu(u^{\D},\sqrt{\D t}) &=&  \mathcal{O}(\sqrt{\D t}), \\\label{wedge}
-\wedge_{\sqrt{\D t},\sqrt{\D t}}(u,v)& =&\mathcal{O}(\sqrt{\D t}) .
\end{eqnarray}
\eqref{nu} follows from the time estimates \eqref{time_cont_multi}.
Now we prove \eqref{wedge}. Let $(x,y)\in \R^2$  and $u^{\D}$ be a piecewise constant function obtained by the numerical scheme.  Consider, 
    \begin{eqnarray*}
    -\wedge^{\D}_T(u^{\Delta},\phi,k^{\D}_{\alpha})
 &=& -\sum\limits_{n=0}^{N-1} \Bigg(\displaystyle\sum_{ij} \Bigg[ \int\limits_{C_i} \int\limits_{C_j} \int\limits_{C^{n}} \left(  \eta_{ij}^n \phi_t(x,y,s)+p_{ij}^n\phi_x(x,y,s)+q_{ij}^n\phi_y(x,y,s)\right) ds dy dx \nonumber\\
    & -&\displaystyle  \int\limits_{C_i} \int\limits_{C_j} \int\limits_{C^{n+1/2}} \left(  \eta_{ij}^{n+1/2} \phi_t(x,y,s)+p_{ij}^{n+1/2}\phi_x(x,y,s)+q_{ij}^{n+1/2}\phi_y(x,y,s)\right) ds dy dx\Bigg] \nonumber\\
    &-&  \int\limits_{C_i} \int\limits_{C_j} \eta_{ij}^0 \phi(x,y,0) dy dx + \sum_{ij}  \int\limits_{C_i} \int\limits_{C_j} \eta_{ij}^N \phi(x,y,T) dy dx \Bigg).
    \end{eqnarray*}  Fundamental theorem of calculus followed by summation by parts imply, $-\wedge^{\D}_T(u^{\Delta},\phi,k^{\D}_{\alpha})$
    \begin{eqnarray*}
    &=& \displaystyle\sum_{ij} \sum\limits_{n=0}^{N-1} \Bigg[ \left( \eta_{ij}^{n+1/2}-\eta_{ij}^{n}\right)  \int\limits_{C_i} \int\limits_{C_j} \phi (x,y,t^{n+1/2})dx dy + \left( p_{ij}^n -p_{i-1,j}^{n} \right)  \int\limits_{C_j}\int\limits_{C^{n}} \phi(x_{i-\frac{1}{2}},y,s) dy ds\\ && \quad \quad \quad \quad \quad\quad \quad \quad \quad \quad \quad \quad \quad \quad \quad\quad \quad \quad \quad \quad+\left( q_{ij}^n -q_{ij-1}^{n} \right)  \int\limits_{C_i}\int\limits_{C^{n}} \phi(x,y_{j-\frac{1}{2}},s)ds dx\Bigg]\\
    &+& \displaystyle\sum_{ij} \sum\limits_{n=0}^{N-1} \Bigg[ \left( \eta_{ij}^{n+1}-\eta_{ij}^{n+1/2}\right)  \int\limits_{C_i} \int\limits_{C_j} \phi (x,y,t^{n+1})dy dx + \left( p_{ij}^{n+1/2} -p_{i-1,j}^{n+1/2} \right)  \int\limits_{C_j}\int\limits_{C^{n+1/2}} \phi(x_{i-\frac{1}{2}},y,s)ds dy\\&& \quad \quad \quad \quad \quad\quad \quad \quad \quad \quad \quad \quad \quad \quad \quad\quad \quad \quad \quad \quad+ \left( q_{ij}^{n+1/2} -q_{ij-1}^{n+1/2} \right)  \int\limits_{C_i}\int\limits_{C^{n+1/2}} \phi(x,y_{j-\frac{1}{2}},s)ds dx\Bigg].
    \end{eqnarray*}
Using the discrete entropy inequalities \eqref{dis_ad_multi_1}-\eqref{dis_ad_multi_2} in the above equation,  $-\wedge^{\D}_T(u^{\Delta},\phi,k^{\D}_{\alpha})$
 \begin{eqnarray*}
    &\leq& \displaystyle \sum_{ij} \sum\limits_{n=0}^{N-1} \Bigg[ \lambda\left( \mP^n_{i+\frac{1}{2},j}-\mP^n_{i-\frac{1}{2},j}\right)  \int\limits_{C_i} \int\limits_{C_j} \phi (x,y,t^{n+1/2})dx dy \\&+& \left( p_{ij}^n -p_{i-1,j}^{n} \right)  \int\limits_{C_j}\int\limits_{C^{n}} \phi(x_{i-\frac{1}{2}},y,s)ds dy + \left( q_{ij}^n -q_{ij-1}^{n} \right)  \int\limits_{C_i}\int\limits_{C^{n}} \phi(x,y_{j-\frac{1}{2}},s)ds dx\Bigg]\\
    &+& \displaystyle\sum_{ij} \sum\limits_{n=0}^{N-1} \Bigg[\lambda \left( \mQ^{n+1/2}_{i,j+\frac{1}{2}}-\mQ^{n+1/2}_{i,j-\frac{1}{2}}\right)  \int\limits_{C_i} \int\limits_{C_j} \phi (x,y,t^{n+1})dy dx \\&+&\left( p_{ij}^{n+1/2} -p_{i-1,j}^{n+1/2} \right)  \int\limits_{C_j}\int\limits_{C^{n+1/2}} \phi(x_{i-\frac{1}{2}},y,s)ds dy + \left( q_{ij}^{n+1/2} -q_{ij-1}^{n+1/2} \right)  \int\limits_{C_i}\int\limits_{C^{n+1/2}} \phi(x,y_{j-\frac{1}{2}},s)ds dx \Bigg],
    \end{eqnarray*}
which on rearrangement implies that  $-\wedge^{\D}_T(u^{\Delta},\phi,k^{\D}_{\alpha})$
\begin{eqnarray*}
 &\leq& 
   \displaystyle\sum_{ij} \sum\limits_{n=0}^{N-1} \Bigg[ \lambda_x  \abs{\mP^n_{i+\frac{1}{2},j} -p_{ij}^{n}}  \int\limits_{C_i} \int\limits_{C_j} \abs{\phi(x+\Delta x,y,t^{n+1/2})-\phi(x,y,t^{n+1/2}) }dy dx\\&+&  \displaystyle\abs{p_{ij}^n-p_{i-1,j}^{n}} \abs{\int\limits_{C^{n}} \int\limits_{C_j}  \phi(x_{i-\frac{1}{2}},y,s) dy ds-\lambda_x \int\limits_{C_j}  \int\limits_{C_i}\phi(x,y,t^{n+1/2})dx dy }
    \\&+&
       \displaystyle\lambda_y  \abs{\mQ^{n+1/2}_{i,j+\frac{1}{2}} -q_{ij}^{n+1/2} }\int\limits_{C_i} \int\limits_{C_j} \abs{\phi(x,y+\Delta y,t^{n+1})-\phi(x,y,t^{n+1}) } dy dx
\\&+&  \displaystyle\left|q_{ij}^{n+1/2}-q_{i,j-1}^{n+1/2}\right| \left| \int\limits_{C_i}\int\limits_{C^{n+1/2}} \phi(x,y_{j-\frac{1}{2}},s)ds dx -\lambda_y  \int\limits_{C_i} \int\limits_{C_j}\phi(x,y,t^{n+1}) dy dx\right|\\&+&   \displaystyle \left( q_{ij}^n -q_{ij-1}^{n} \right)  \int\limits_{C_i}\int\limits_{C^{n}} \abs{\phi(x,y_{j-\frac{1}{2}},s)}ds  dx +\left| p_{ij}^{n+1/2} -p_{i-1,j}^{n+1/2} \right|  \int\limits_{C_j}\int\limits_{C^{n+1/2}} \abs{\phi(x_{i-\frac{1}{2}},y,s)}ds dy\Bigg].
     \end{eqnarray*}
Adding and subtracting \begin{eqnarray*}
&&\lambda_x\int_{C_i}\phi(x_{i-\frac{1}{2}},y,t^{n+1/2}) dx=\int_{C^n}\phi(x_{i-\frac{1}{2}},y,t^{n+1/2}) dt\\
\text{and }&&\lambda_y\int_{C_j}\phi(x,y_{j-\frac{1}{2}},t^{n+1}) dx=\int_{C^{n+1/2}}\phi(x,y_{j-\frac{1}{2}},t^{n+1}) dt   \end{eqnarray*} in the terms \begin{eqnarray*}
&&\left|\displaystyle\int_{C^{n}}\phi(x_{i-\frac{1}{2}},y,s)ds-\lambda_x\int_{C_i}\phi(x,y,t^{n+1/2})dx\right|\\
\text{and } &&
\left|\displaystyle\int_{C^{n+1/2}}\phi(x,y_{j-\frac{1}{2}},s)ds-\lambda_y\int_{C_j}\phi(x,y,t^{n+1})dy\right|
\end{eqnarray*} respectively, we get
\begin{eqnarray*}
     -\wedge^{\D}_T(u^{\D},\phi,k^{\D}_{\alpha})&\leq& 
   \displaystyle\sum_{ij} \sum\limits_{n=0}^{N-1} \Bigg[\lambda_x \mG_1^{\phi} \abs{\mP^n_{i+\frac{1}{2},j} -p_{ij}^{n}} +\left|p_{ij}^n-p_{i-1,j}^{n}\right|  \left(\mG_2^{\phi}+\lambda_x \mG_3^{\phi}\right)\\&&+ 
   \lambda_y G_4^{\phi} \left|\mQ^{n+1/2}_{i,j+\frac{1}{2}} -q_{ij}^{n+1/2}\right| +\left(q_{ij}^{n+1/2}-q_{i,j-1}^{n+1/2}\right) (\mG_5^{\phi}+\lambda_y \mG_6^{\phi})
    \\&&+ \left( q_{ij}^n -q_{ij-1}^{n} \right) \mG_7^{\phi}+\left( p_{ij}^{n+1/2} -p_{i-1,j}^{n+1/2} \right) \mG_8^{\phi}\Bigg]\\&:=& 
    \displaystyle\sum_{ij} \sum\limits_{n=0}^{N-1} \Bigg[\mG_1^{\phi}K_1^{ijn}+\mG_2^{\phi}K_2^{ijn}+\ldots+\mG_8^{\phi}K_8^{ijn}\Bigg].
\end{eqnarray*}
where,
 \begin{eqnarray*}
\mG_1^{\phi}&=& \int\limits_{C_i} \int\limits_{C_j} \phi(x+\Delta x,y,t^{n+1/2})-\phi(x,y,t^{n+1/2}) dy dx,\\
\mG_2^{\phi}&=& \int\limits_{C_i}\int\limits_{C^{n}} \abs{\phi(x_{i-1/2},y,t)-\phi(x_{i-1/2},y,t^{n+1/2})} ds dx,\\
\mG_3^{\phi}&=&  \int\limits_{C_i} \int\limits_{C_j} \abs{\phi(x,y,t^{n+1/2})-\phi(x_{i-1/2},y,t^{n+1/2})}dy dx,\\
\mG_4^{\phi}&=& \int\limits_{C_i} \int\limits_{C_j} \phi(x,y+\D y ,t^{n+1})-\phi(x,y,t^{n+1}) dy dx, \\
\mG_5^{\phi}&=& \int\limits_{C_j}\int\limits_{C^{n+1/2}} \abs{\phi(x_{i-1/2},y,t)-\phi(x_{i-1/2},y,t^{n+1})} ds dy, \\
\mG_6^{\phi}&=&  \int\limits_{C_i} \int\limits_{C_j} \abs{\phi(x,y,t^{n+1})-\phi(x,y_{j-1/2},t^{n+1})} dy dx,\\
\mG_7^{\phi}&=& \int\limits_{C_i}\int\limits_{C^{n}} \phi(x,y_{j-\frac{1}{2}},s)ds dx,\\
\mG_8^{\phi}&=& \int\limits_{C_j}\int\limits_{C^{n+1/2}} \phi(x_{i-\frac{1}{2}},y,s) dy ds.
\end{eqnarray*}For each $(\overline{x},\overline{y},s)\in \Pi_T,$ consider the test function $\phi(x,y,t):=\Phi^{\sqrt{\D t},\sqrt{\D t}}(x,y,t,\overline{x},\overline{y},s)$ and $\alpha=\beta(\overline{x},\overline{y},v(\overline{y},s)).$

Using the properties of $\Phi^{\sqrt{\D t},\sqrt{\D t}},$ the following estimate can be obtained (see \cite{HRbook} for the details). 

\begin{eqnarray}\label{Gl}
\displaystyle \int\limits_{\Pi_T}\displaystyle \mG_l^{\Phi^{\sqrt{\D t},\sqrt{\D t}}(\cdot,\cdot,\cdot,\overline{x},\overline{y},s)}d\overline{x} d\overline{y} ds=\mathcal{O}({\Delta t})^{5/2}, \quad l=1,2,\ldots,8.
\end{eqnarray}

Our assumptions on the flux \eqref{C1}-\eqref{C2} imply the following:
\begin{eqnarray*}
\label{pp}|p_{ij}^{n/2}-p_{i-1,j}^{n/2}| &\leq& C \left[|u_{ij}^{n/2}-u_{i-1,j}^{n/2}|+ |r_{ij}-r_{i-1,j}| \right],\\
\label{pP}|p_{ij}^n-\mP^n_{i,j+\frac{1}{2}}| &\leq& C\sum_{k=-1}^{1}|u_{i+k,j}^{n/2}-u_{ij}^{n/2}|,\\
\label{qq}|q_{ij}^{n/2}-q_{i,j-1}^{n/2}| &\leq& C \left[|u_{ij}^{n/2}-u_{i,j-1}^{n/2}|+ |r_{ij}-r_{i,j-1}| \right],\\
\label{qQ}|q_{ij}^{n/2}-\mQ^{n/2}_{i,j+\frac{1}{2}}| &\leq& C\sum_{k=-1}^{1}|u_{i,j+k}^{n/2}-u_{ij}^{n/2}|.
\end{eqnarray*}
Since the numerical approximations are uniformly total variation bounded, the above inequalities imply that, $\D t\sum\limits_{ij}K_l^{ijn}$ is uniformly bounded for $l\in \{1,2,\ldots,8\},$  $n=0,1,2,\ldots,N-1, \alpha \in \R \text{ and } \D >0.$

Now \eqref{Gl} implies the following
\begin{eqnarray}
\nonumber \Lambda^{\D}_{\sqrt{\D t}, \sqrt{\D t}}(u^{\D},v) &=& \int_{\Pi_T}\wedge^{\D}_T(u^{\D},\Phi^{{\sqrt{\D t}, \sqrt{\D t}}}(\cdot,\cdot,\cdot,\overline{x},\overline{y},s),k^{\D}_{\beta(y,v(y,s))})\\\nonumber&=&  \left(\sum\limits_{ij} \sum\limits_{n=0}^{N-1} K_l^{ijn}   \right)
\int_{\Pi_T} G_l^{\Phi^{\sqrt{\D t},\sqrt{\D t}}(\cdot,\cdot,\cdot,\overline{x},\overline{y},s)}d\overline{x} d\overline{y} ds \\\label{wedgeD}
&=& \mathcal{O}(\D t^{-2}) \mathcal{O}(\D t^{5/2})= \mathcal{O}(\sqrt{\D t}).
\end{eqnarray}

Note that,\begin{eqnarray*}
\wedge^{\D}_T(u^{\D},\phi,k^{\D}_{\alpha})&=&
    \int_{\Pi_T}\Big( |u^{\D}(t,\vb{x})-k^{\D}_\alpha(\vb{x})|\phi_{t}\\&&+\sum\limits_{i=1}^2\sgn (u^{\D}(t,\vb{x})-k^{\D}_\alpha(\vb{x})) \Big(g_i(\beta^{\D}(\vb{x},u^{\D}(t,\vb{x})))-\alpha\Big)\phi_{x_i}\Big)  d\vb{x} d{t} \\&&-\int_{\R^2}|u^{\D}(T,\vb{x})-k^{\D}_\alpha(\vb{x})|\phi(T,\vb{x}) d\vb{x}+\int_{\R^2}|u_0(\vb{x})-k^{\D}_\alpha(\vb{x})|\phi(0,\vb{x}) d\vb{x}.
\end{eqnarray*}Thus, we have,
\begin{eqnarray*}\label{lambdelta}
\abs{\Lambda_{\sqrt{\D t}, \sqrt{\D t}}(u^{\D},v)-\Lambda^{\D}_{\sqrt{\D t}, \sqrt{\D t}}(u^{\D},v)} &\leq& C\Bigg(\int\limits_{{\Pi_T}^2} \abs{k_{\alpha}-k^{\D }_{\alpha}} \left| \Phi_t^{{\sqrt{\D t}, \sqrt{\D t}}}+\Phi_x^{{\sqrt{\D t}, \sqrt{\D t}}}+\Phi_y^{{\sqrt{\D t}, \sqrt{\D t}}}\right|\\&&\quad\quad+
\int\limits_{\R^4}\abs{k_{\alpha}-k^{\D }_{\alpha}} \left|\Phi^{{\sqrt{\D t}, \sqrt{\D t}}}(\cdot,0)+\Phi^{{\sqrt{\D t}, \sqrt{\D t}}}(\cdot,T)\right|\Bigg).
\end{eqnarray*}
Since, $\norma{k-k^{\D}}_{L^1(\R^2)}=\mathcal{O}(\D t),$ using \eqref{S5} and   \eqref{wedgeD} in the above inequality, we get
\begin{eqnarray*}\Lambda_{\sqrt{\D t}, \sqrt{\D t}}(u^{\D},v)=\Lambda^{\D}_{\sqrt{\D t}, \sqrt{\D t}}(u^{\D},v)+ \mathcal{O}(\sqrt{\D t})=\mathcal{O}(\sqrt{\D t}).\end{eqnarray*}This completes the proof of the theorem.
\end{proof}
\begin{remark}\normalfont
In Theorem~\ref{CR} we proved that the rate of convergence is not less than 1/2. This result has to be considered as the worst case estimate in the sense that rate cannot be less than 1/2.  An example due to Sabac shows that in general this result cannot be improved as the rate $1/2$ is achieved for the example. However the method in many cases exhibits rates much higher than $1/2$ as shown in the next section.
\end{remark}

\section{Numerical Simulations} \label{sec_numerical}
 This section displays the performance of the numerical scheme, for the multidimensional analogues of the Example 4.1 and 4.2 of \cite{GTV_2020_2}. Numerical experiments are performed on the spatial domain $[0,6]\times[0,6]$ with $M=50, 100, 200$ and $400$ uniformly spaced spatial grid points along the $x$ and $y$ directions. It will be seen that the scheme is able to capture the expected solutions well,  as in the case of one dimension.
\begin{example}\label{ex:rare}\normalfont
We consider the IVP \eqref{eq:discont}-\eqref{eq:data} with fluxes as defined below:
\begin{eqnarray*}
A_i(x,y,u)&:=&g_i(u+r(x)), \quad \text{for }i=1,2.
\end{eqnarray*}
\begin{eqnarray}
g_1(u)=u^2/2,\quad g_2(u)=sin(u) \quad \text{ and } \quad
r(x)= 
\begin{cases}
p, \quad & x <1,\\
pq^{n-1}, \quad & x \in C_n,n\in \mathbb{N},\\
 0, \quad & x>a_{\infty},
\end{cases}
\end{eqnarray}
where $p=4, q=0.8$ and for each $n\in \mathbb{N}$, $C_n=[a_n,a_{n+1}]$, with
\begin{eqnarray*}
a_1=1 \text{ and }a_n=1+\sum_{i=1}^{n-1} \tilde{a}_i \text{ for }n\geq 2
\end{eqnarray*}
with
\begin{equation*}
\tilde{a}_n= 
\begin{cases}
pq^{n-1}-pq^n, \quad & \text{ if } n \text{ is odd},\\
pq^{n-2}-pq^{n-1}, \quad & \text{ if } n \text{ is even}.
\end{cases}
\end{equation*}
Define 
and consider a piecewise constant initial data 
\begin{equation}
u_0(x,y)= 
\begin{cases}
-pq, \quad & x <a_2,\\
-pq^n, \quad & x \in C_n \text{ and } n \text{ odd},\\
-pq^{n-2}, \quad & x \in C_n \text{ and } n \text{ even},\\
0, \quad & x>a_{\infty}.
\end{cases} 
\end{equation}At $t=1,$ the solution  is given by, 
\begin{equation}\label{exact}
u(1,x,y)= 
\begin{cases}
-pq, \quad & x <a_2,\\
x-a_{n}-pq^{n-1}, \quad & x \in C_{n} \text{ and } n \text{ odd},\\
x-a_{n+1}-pq^{n-1}, \quad & x \in C_{n} \text{ and } n \text{ even},\\
0, \quad & x>a_{\infty}.
\end{cases} 
\end{equation}
Figure \ref{fig:rare} plots the numerical solutions at the final time $t=1$ for the mesh size $\D x= \D y= 6/200$. It can be seen that the scheme captures both stationary shocks and rarefactions efficiently.
\begin{figure}[H]
 \centering
 \includegraphics[width=\textwidth,keepaspectratio]{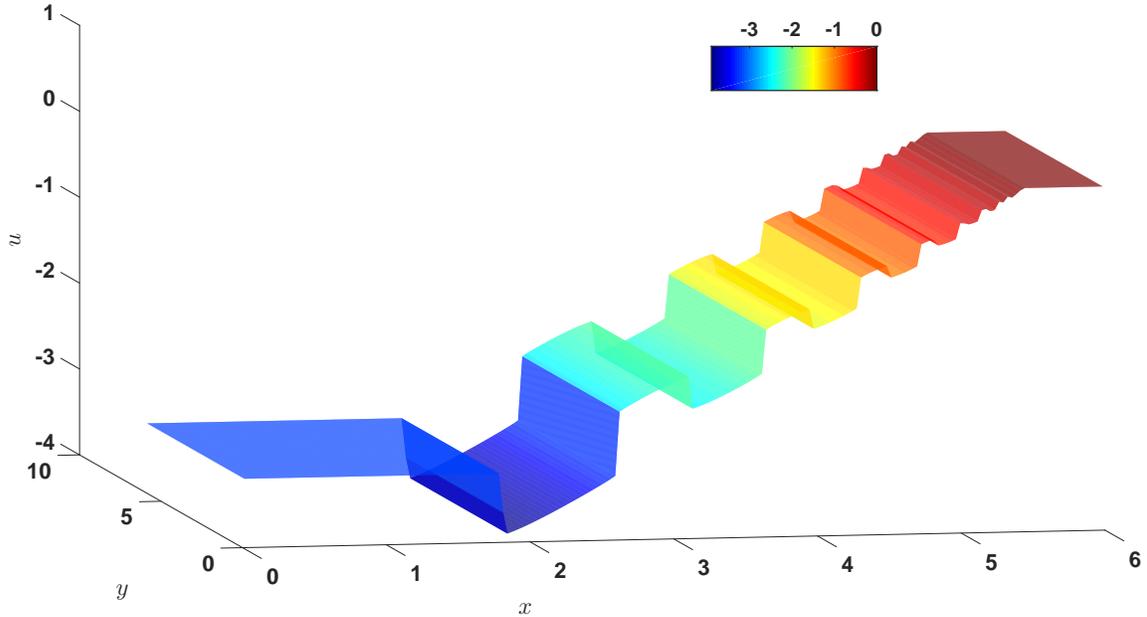}
 \caption{Example~\ref{ex:rare}. The solution at $t=6$ with mesh size $\D x=\D y=6/200.$ Solution contains infinitely many shocks along the spatial discontinuities of the flux, which accumulates along the plane $x = 5.$}
\label{fig:rare}
\end{figure}

\begin{table}[H]
     \centering
    \begin{tabular}{ |c|c|c|c|c|c| } 
     \hline
 M&  $e_{\Delta }$ &$\TV(u^{\Delta }(\cdot,1))$&$\TV(\beta(\cdot,u^{\Delta }(1,\cdot))$\\\hline
 50& 1.3464 & 32.9298&33.9876\\\hline
 100&  0.9618 &34.4796& 35.9166\\\hline
 200&0.6282&37.7934 &40.0374\\\hline
 400&0.4038&39.4704& 41.8296\\\hline
     \end{tabular}
     \caption{ Approximate $L^1$ error and total variation at $t=1$ for Example \ref{ex:rare}.}
   \label{table:1}
 \end{table}
Clearly, the solutions are the extensions of the solutions obtained in the one dimensional case (see Example 4.1, \cite{GTV_2020_2}), more precisely $u(1,x,y)=u(1,x)$, for $(x,y) \in [0,6] \times [0,6].$ Thus, the values listed in the above table are approximately six times of those obtained in the corresponding 1D simulations (see Table 1, \cite{GTV_2020_2}). 
\end{example}
\begin{example}\normalfont\label{ex:shock}
We consider the IVP \eqref{eq:discont}-\eqref{eq:data} with $u_0(x,y)=2$ and fluxes as defined below:
\begin{eqnarray*}
A_i(x,y,u)&:=&g_i(u+r(x)) \quad \text{for }i=1,2
\end{eqnarray*}
where,
\begin{equation*}
g_1(u) = 
\begin{cases}
-u-1, \quad & u <-1,\\
0, \quad & u \in (-1,0),\\
u, \quad & u>1,\\
\end{cases}\quad
g_2(u)=\sin(u) \quad \text{ and } \quad r(x)= 
\begin{cases}
2, \quad & x <1,\\
r_n \chi_{[a_n,a_{n+1}]}(x), \quad & x \in (1,5),\\
1, \quad & x>5,\\
\end{cases}
\end{equation*}
with \begin{eqnarray*}
a_n=5(1-0.8^n),r_n=1-(-0.8)^n.
\end{eqnarray*} 
The flux considered here admits infinitely many spatial discontinuities which accumulates along the plane $x=5.$ Solution at $t=6$ is given by,
\begin{equation*}
u(6,x,y)= r(x) \quad \text{ for } (x,y)\in [0,6]\times[0,6].
\end{equation*}
Figure \ref{fig:shock} plots the numerical solutions at the final time $t=6$ for the mesh size $\D x= \D y= 6/200$. It can be seen that the scheme captures both stationary shocks efficiently.
\begin{figure}[H]
 \centering
 \includegraphics[width=\textwidth,keepaspectratio]{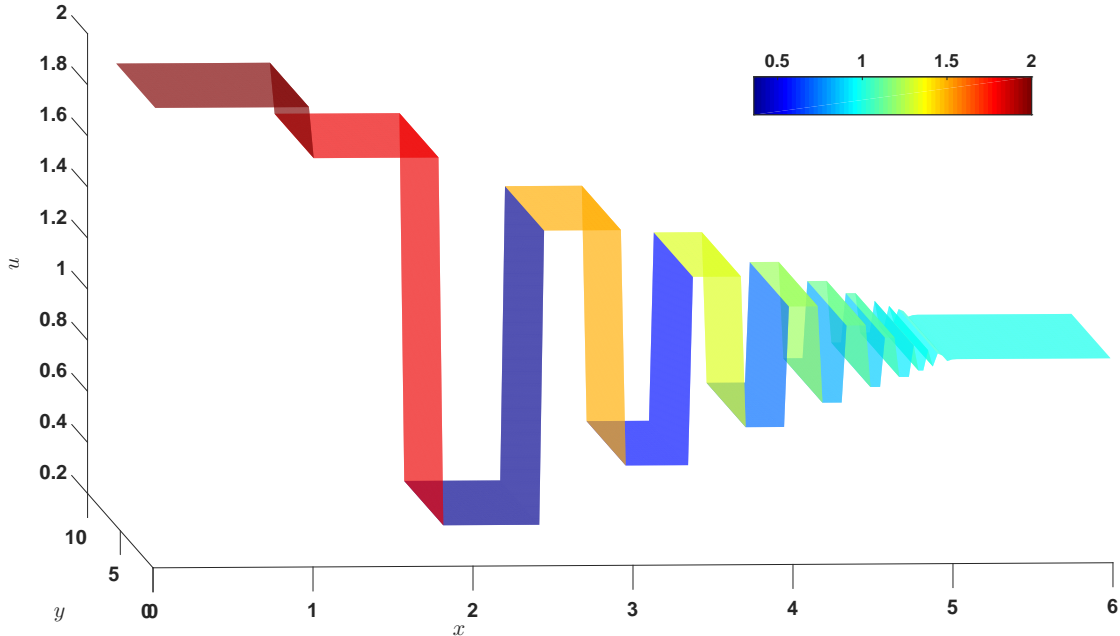}
 \caption{Example~\ref{ex:shock}. The solution at $t=6$ with mesh size $\D x=\D y=6/200.$ Solution contains infinitely many shocks along the spatial discontinuities of the flux, which accumulate along the plane $x = 5.$}
\label{fig:shock}
\end{figure}

\begin{table}[H] 
     \centering
    \begin{tabular}{ |c|c|c|c|c|c| } 
     \hline
 M&  $e_{\Delta }$& $\TV(u^{\Delta }(6,\cdot))$&$\TV(\beta(\cdot,u^{\Delta }(6,\cdot))$\\\hline
 50&2.7933e-02& 40.701& 8.7198e-02\\\hline
 100& 2.559e-03& 41.914& 1.3788e-02\\\hline
 200&1.1147e-04& 43.4088 &  1.07436e-03\\\hline
 400&3.5146e-07&43.6824&  6.3834e-06\\\hline
     \end{tabular}
     \caption{ Approximate $L^1$ error and total variation at $t=6$ for Example \ref{ex:shock}}.
  \label{table2}
 \end{table}
 As in the previous example, the values listed in the above table are approximately six times of those obtained in the corresponding 1D simulation  (see Table 2, \cite{GTV_2020_2}).
\end{example}
\textbf{Acknowledgement.} First and last authors, would like to thank Department of Atomic Energy, Government of India, under project no. \text{12-R\&D-TFR-5.01-0520.} First author would also like to acknowledge Inspire faculty-research grant \text{DST/INSPIRE/04/2016/000237}. 


\end{document}